\documentclass[smallextended,nospthms]{svjour3}
\usepackage{amsmath, amssymb, amscd}
\usepackage[mathscr]{eucal}
\newcommand{\C}{\mathbb{C}}
\newcommand{\Z}{\mathbb{Z}}

\newtheorem{thm}{Theorem}
\newtheorem{defn}[thm]{Definition}
\newtheorem{prop}[thm]{Proposition}
\newtheorem{coro}[thm]{Corollary}
\newtheorem{lem}[thm]{Lemma}

\newtheorem{rem}[thm]{Remark}

\newcommand{\Li}{\mathrm{Li}} 
\newcommand{\re}{\mathop{\mathrm{Re}}} 
\newcommand{\im}{\mathop{\mathrm{Im}}} 
\def\mod{\operatorname{mod}}
\newenvironment{proof}{\par\noindent{\bf Proof.}}{$\square$\par\bigskip}

\begin{document}

\title{Higher Mahler measure for cyclotomic polynomials and Lehmer's
question\thanks{M.L. was supported by NSERC Discovery Grant 355412-2008 and a
Faculty of Science Startup grant from the University of Alberta. K.S. was
supported by a Pacific Institute for the Mathematical Sciences Postdoctoral
Fellowship and the above grants.}}

\author{Matilde Lal\'in \and Kaneenika Sinha}

\institute{M. Lal\'in \at
D\'epartement de math\'ematiques et de statistique,\\ Universit\'e de Montr\'eal,\\
Montr\'eal, QC, H3C3J7, Canada\\ 
              Tel.: 1-514-343-6689\\
              Fax: 1-514-343-5700\\
              \email{mlalin@dms.umontreal.ca}  
              \and
              K. Sinha \at
              Department of Mathematical Sciences, \\
              Indian Institute of Science Education and Research Kolkata,\\
              Mohanpur, Nadia, 741252, West Bengal, India\\
              Tel.: 91-33-2587-3223\\
              Fax: 91-33-2587-3019\\
              \email{kaneenika@iiserkol.ac.in}
}
\date{}
\maketitle

\begin{abstract}
The $k$-higher Mahler measure of a nonzero polynomial $P$ is the integral of
$\log^k|P|$ on the unit circle. In this note, we consider Lehmer's question
(which is a long-standing open problem for $k=1$) for $k>1$ and find some
interesting formulae for 2- and 3-higher Mahler measure of cyclotomic
polynomials.
\end{abstract}

\section{Introduction}

\begin{defn}
Given a non-zero polynomial $P(x)\in \C[x]$ and a positive integer $k,$ the
$k$-higher Mahler measure of $P$ is defined by 
\[m_k(P) := \frac{1}{2\pi i}\int_{|x|=1}\log^k|P(x)|\frac{dx}{x},\]
or, equivalently, by
\[m_k(P) := \int_{0}^1\log^k|P(e^{2\pi i \theta})|d\theta.\]
\end{defn}

We observe that for $k=1$, $m_1(P)$ is the classical (logarithmic) Mahler
measure given by 
\[m(P) := \log|a| +\sum_{j=1}^n\log^{+}|r_j|, \mbox{  for } P(x) =
a\prod_{i=1}^n(x-r_j)\]
where $\log^+ t = \log\text{max}\{1,t\}$ for a non-negative real number $t.$
This object first appeared in a 1933 paper by Lehmer \cite{Lehmer} in connection
with a method for constructing large prime numbers. A generalization to
multivariable polynomials appeared in a work by Mahler \cite{Mahler} (who was
interested in tools for transcendence theory) about 30 years later. The
generalization to higher Mahler measures was recently considered in \cite{KLO}
for the first time.

Higher Mahler measures of polynomials are usually very hard
to compute, even for simple linear polynomials in one variable.  However,  the
investigation carried out in \cite{KLO} reveals direct connections between these
measures and special values of zeta functions and polylogarithms. In the case of the classical Mahler measure, 
analogous relations with special values of $L$-functions have been explained by Deninger \cite{D} and others 
in terms of evaluations of regulators in the context of Beilinson's conjectures. One of the motivations
for considering higher Mahler measures (in addition to classical Mahler measures) is that they yield different
periods from the ones that we obtain from the usual Mahler measure thus revealing a more complicated structure for the regulator (see \cite{Lal} for more details).

One of the tools for studying general $k$-higher Mahler measures is the
following:
\begin{defn}
For a finite collection of non-zero polynomials $P_1,\dots, P_l \in \C[x],$
their multiple  Mahler measure is defined by
\[m(P_1,\dots, P_l) :=\frac{1}{2\pi i}\int_{|x| = 1}\log|P_1(x)|\dots
\log|P_l(x)| \frac{dx}{x}.\]
\end{defn}

Our main interest in this note is the case of $P(x) \in \Z[x]$ but we consider other cases as well, such as products of cyclotomic polynomials. We recall the following well-known theorem of
Kronecker \cite{K}:
\begin{thm}\label{Kronecker} Let $P(x) = \prod_{j=1}^n(x-r_j) \in \Z[x].$ If
$|r_j| \leq 1$ for each $j,$ then the $r_j$'s are zero or roots of unity.
\end{thm}

An immediate consequence of Kronecker's theorem is that for a non-zero
polynomial $P(x) \in \Z[x],$ $m(P) = 0$ if and only if $P$ is monic and is a
product of powers of $x$ and cyclotomic polynomials. 

Lehmer \cite{Lehmer} asked the following question: {\em Given $\epsilon >0$, can
we find a polynomial $P(x) \in \Z[x]$ such that $0<m(P)<\epsilon$? }

This question is still open\footnote{See \cite{S} for a recent general survey on
the status of this problem.}. The smallest known measure greater than 0 is that
of a polynomial that he found in his 1933 paper:
\[m(x^{10}+x^9-x^7-x^6-x^5-x^4-x^3+x+1) = 0.1623576120\dots .\]
A polynomial $P(x)$ is said to be {\it reciprocal} if $P(x)=\pm x^dP\left(x^{-1}
\right)$ where $d=\deg P$. Notice that the above polynomial is reciprocal. Lehmer's question was answered negatively by Breusch \footnote{Later Smyth worked on this problem independently in
\cite{S1} and found the best possible
constant.} in \cite{Br} for nonreciprocal polynomials. 

Lehmer's question has attracted considerable attention in the last few decades,
as it has connections beyond number theory, such as entropies of dynamical
systems and to polynomial knot invariants.  

In this note, we explore the analogue of Lehmer's question for $m_{k}$ for
$k>1.$  We investigate lower bounds and limit points for higher Mahler measures
and the value of $m_2$ and $m_3$ at cyclotomic polynomials.

Our main results are the following:
\begin{thm}\label{maintheorem1} If $P(x) \in \Z[x]$ is not a monomial, then for
any $h\geq 1,$
\[m_{2h}(P) \geq 
\begin{cases}
\left(\frac{\pi^2}{12}\right)^h,&\text{ if } P(x)\text{ is reciprocal,}\\
\left(\frac{\pi^2}{48} \right)^h,&\text{ if } P(x)\text{ is non-reciprocal}. 
\end{cases}\]
\end{thm}
This theorem is significant because the lower bound it provides is general and
unconditional.  Unlike well-known results regarding the lower bound for $m(P),$
the above theorem is not restricted by the behavior of the coefficients,
degrees, or the reducibility properties of $P(x).$ In particular, this result
implies that Lehmer's question has a negative answer for $m_{2h}$. 

A careful study of the proof of Theorem \ref{maintheorem1} reveals that $m_2(P)$
for $P$ reciprocal is minimized when $P(x)$ is a product of monomials and
cyclotomic polynomials.  Therefore, it is of interest to find out explicit
values of 2-higher Mahler measures of cyclotomic polynomials.  In this
direction, we prove the following theorem.

\begin{thm}\label{m2-cyclotomic}
For a positive integer $n,$ let $\phi_n(x)$ denote the $n$-th cyclotomic
polynomial and $\varphi$ Euler's function.  Then
\[m(\phi_{m}(x),\phi_{n}(x))=\frac{\pi^2}{12}\frac{(m,n)\varphi([m,n])(-1)^{
r(m)+r(n)}2^{r((m,n))}}{[m,n]^2} \prod_{p\mid mn, p\nmid (m,n)} p,\]
where $r(x)$ denotes the number of distinct prime divisors of $x$ and the
product is taken over prime numbers $p$.  In particular, for $m=n,$ we get 
\[m_2(\phi_{n}(x))=\frac{\pi^2}{12}\frac{\varphi(n)2^{r(n)}}{n}.\]
\end{thm}
This theorem allows us to compute $m_2(P)$ for $P$ any product of cyclotomic
polynomials.  This naturally leads us to investigate the 3-higher Mahler measure
of such polynomials.  We therefore prove the following theorem which relates
$m_3(P)$ to $\zeta(3)$ and the polylogarithm.

\begin{thm} \label{product-cyclotomic-m3}
If $P(x)$ has all its roots on the unit circle, in other words, if $P(x)$ has
the form 
\begin{equation*}\label{eq:prod}
P(x)=\prod_{j=1}^n(x-e^{2\pi i\alpha_j}),
\end{equation*}
with $0 \leq \alpha_1\leq \dots \leq \alpha_n < 1$, then
\begin{eqnarray*}
m_3(P)
&=&-\frac{3}{2}n^2 \zeta(3) -3n \sum_{1\leq k<l\leq n} C_3(2 \pi
(\alpha_l-\alpha_k))\\&&-3\pi \sum_{1\leq k<l\leq n}  S_2(2 \pi
(\alpha_l-\alpha_k)) \left(n(\alpha_l-\alpha_k) -(l- k)\right),
\end{eqnarray*}
where
\[C_\ell(t)=\sum_{n=1}^\infty \frac{\cos(nt)}{n^\ell} \quad \mbox{ and } \quad
S_\ell(t)=\sum_{n=1}^\infty \frac{\sin(nt)}{n^\ell}\]
are the Clausen functions given by real and imaginary parts of the classical
polylogarithm $\Li_\ell\left(e^{2\pi i t}\right)$ defined by
\[\Li_\ell(z) = \sum_{n=1}^\infty \frac{z^n}{n^\ell}\]
in the unit disk.
\end{thm}

Lehmer's question can be rephrased as whether $0$ is a limit point for values of
$m$. We generalize Lehmer's question by asking if $0$ is a limit point for
values of $m_{2k+1}$ for $k\geq 1.$  In this context, we prove the following.

\begin{thm}\label{m_3-limit-0}
Let $P_n(x) = \frac{x^n-1}{x-1}$. For $h\geq 1$ fixed,
\[\lim_{n\rightarrow \infty} m_{2h+1}(P_n)=0.\]
Moreover, this sequence is nonconstant.
\end{thm}
We obtain, in this way, a positive answer for Lehmer's
question for $m_{2h+1}$. 

 Section \ref{sec:m2-cyclo-bound} contains a proof of Theorem
\ref{maintheorem1}, which relies upon a lower bound for $m_2$ for products of
cyclotomic polynomials.  We obtain some explicit formulae for $m_2$ for
cyclotomic polynomials and their products in Section
\ref{sec:m2-cyclo-explicit}, thereby proving Theorem \ref{m2-cyclotomic}. 
Section \ref{sec:m3} contains some partial results towards $m_3$ for cyclotomic
polynomials.  In particular we prove Theorem \ref{product-cyclotomic-m3} in this
section. Section \ref{sec:limit} presents results about limiting points for
$m_k$.  We first consider an $m_3$-version for Theorem \ref{m_3-limit-0} in
\ref{BLm_3}.  A fundamental ingredient in the proof of Theorem \ref{m_3-limit-0}
is a theorem of Boyd and Lawton which shows that the Mahler measure of a
multivariable polynomial  arises as a limit of Mahler measures of polynomials of
one variable.  In \ref{sec:Boyd-Lawton} we discuss a generalization of 
Boyd--Lawton theorem and prove the limit of Theorem \ref{m_3-limit-0}. In Section 5.3 we prove that these sequences are non identically zero.  Finally, Section
\ref{sec:table} includes a discussion about future questions and a table with
values of $m_2(P)$ for the reciprocal non-cyclotomic polynomials $P$ of degree
less than or equal to 14 and $m(P) < 0.25$.  We observe that all the polynomials
in the table have lower values of $m_2$ than Lehmer's degree 10 polynomial.

\section{A lower bound for $2h$-Mahler measures}\label{sec:m2-cyclo-bound}
In this section, we prove Theorem \ref{maintheorem1}. In order to do that, we
first find a lower bound for $m_2$ of products of cyclotomic polynomials. 
\begin{thm} \label{product-cyclotomic}
If $P(x)$ is a product of cyclotomic polynomials and monomials, but is not a
single monomial, then 
\[m_2(P) \geq \frac{\pi^2}{12}.\]
\end{thm}

Before proving this, recall the following theorem from \cite{KLO} (Theorem 7):
\begin{thm}\label{Theorem-KLO} 
For $0\leq \alpha \leq 1,$ 
\[m(1-x,1-e^{2\pi i \alpha}x ) = \frac{\pi^2}{2}\left(\alpha^2-\alpha+\frac16
\right).\]
\end{thm}

We also need the following property:
\begin{lem} \label{lem:m2-cyclotomic} If $P(x)$ has all its roots on the unit
circle, in other words, if $P(x)$ has the form 
\begin{equation*}
P(x)=\prod_{j=1}^n(x-e^{2\pi i\alpha_j}),
\end{equation*}
with $0\leq \alpha_j < 1$, then
\begin{eqnarray*}
m_2(P)&=& \frac{\pi^2}{2} \sum_{1\leq j,k\leq n}
\left((\alpha_j-\alpha_k)^2-|\alpha_j-\alpha_k|+\frac16 \right).
\end{eqnarray*}
\end{lem}

\begin{proof}
By applying Theorem \ref{Theorem-KLO}, we can express $m_2(P)$ in terms of the
arguments $\alpha_i$:
\begin{eqnarray*}
m_2(P)&=&\sum_{1\leq j,k\leq n} m(1-e^{2\pi i \alpha_j}x,1-e^{2\pi i \alpha_k}x
)=\sum_{1\leq j,k\leq n} m(1-x,1-e^{2\pi i |\alpha_j-\alpha_k|}x )\\ &=&
\frac{\pi^2}{2} \sum_{1\leq j,k\leq n}
\left((\alpha_j-\alpha_k)^2-|\alpha_j-\alpha_k|+\frac16 \right).
\end{eqnarray*}
\end{proof}

\begin{proof}[Theorem \ref{product-cyclotomic}]:
Since $\log|x|=0$ on the unit circle, the monomial factors do not change the
value of $m_2(P)$. Thus, we may assume that $P(x)$ can be written as
\begin{equation*}
P(x)=(x-1)^a(x+1)^b\prod_{j=1}^{2n}(x-e^{2\pi i\alpha_j})
\end{equation*}
with $0\leq \alpha_1\leq \dots \leq \alpha_{2n}\leq 1$ with
$\alpha_j=1-\alpha_{2n+1-j}$. In addition, $a,b \in \{0,1\}$ as they account for
the fact that we may have an odd number of factors $x-1$ and/or $x+1$ in the
product.  Using  that $m(x+1,x-1)=-\frac{\pi^2}{24}$ and Lemma
\ref{lem:m2-cyclotomic}, we obtain
\begin{eqnarray*}
m_2(P) &=& am_2(x-1)+bm_2(x+1)+2abm(x+1,x-1)\\
&&+ 2a m\left(x-1, \prod_{j=1}^{2n}(x-e^{2\pi i\alpha_j})\right)
+2 b m\left(x+1, \prod_{j=1}^{2n}(x-e^{2\pi i\alpha_j})\right)\\
&&+m_2\left(\prod_{j=1}^{2n}(x-e^{2\pi i\alpha_j})\right)\\
&=&\frac{\pi^2}{2}\left(\frac{a+b-ab}{6} +2a \sum_{j=1}^{2n}
\left(\alpha_j^2-\alpha_j+\frac16 \right)\right.\\
&&\left.+2b \sum_{j=1}^{2n}
\left(\left(\alpha_j-\frac{1}{2}\right)^2-\left|\alpha_j-\frac{1}{2}
\right|+\frac16 \right)+  4 n\sum_{j=1}^{2n} \alpha_j^2\right.\\
&&\left.- \sum_{1\leq j,k\leq 2n}
\left(2\alpha_j\alpha_k+|\alpha_j-\alpha_k|\right)+\frac{2n^2}{3}\right)\\
&=&\frac{\pi^2}{2}\left(\frac{a+b-ab}{6} +2(a+b) \sum_{j=1}^{2n}
\alpha_j^2-2a\sum_{j=1}^{2n} \alpha_j -4b  \sum_{j=n+1}^{2n}
\alpha_j+bn\right.\\
&&\left. +   4 n\sum_{j=1}^{2n} \alpha_j^2
- 2\sum_{1\leq j,k\leq 2n} \alpha_j\alpha_k-2\sum_{j=1}^{2n} j
\alpha_j+2\sum_{j=1}^{2n} (2n+1-j) \alpha_j\right.\\
&&\left.+\frac{2n(n+a+b)}{3}\right).\\
\end{eqnarray*}
Because of $\alpha_j=1-\alpha_{2n+1-j}$, we have that $\sum_{j=1}^{2n}
\alpha_j=n$.  This implies that
\begin{eqnarray*}
m_2(P)&=&\frac{\pi^2}{2}\left( 2(a+b+2n) \sum_{j=1}^{2n} \alpha_j^2 -4b 
\sum_{j=n+1}^{2n} \alpha_j-4\sum_{j=1}^{2n} j \alpha_j\right.\\
&&\left. +2n(n+1) +\frac{n(2n-4a+5b)}{3}+\frac{a+b-ab}{6}\right).\\
\end{eqnarray*}

Let $\alpha:=\alpha_j$ with $1\leq j \leq n $ so that $0\leq
\alpha\leq\frac{1}{2}$. In this case, we define
\begin{eqnarray*}
 g(\alpha)&:=& 2(a+b+2n) (\alpha^2+(1-\alpha)^2) - 4 (j
\alpha+(2n+1-j+b)(1-\alpha))\\
&=& 4(a+b+2n)\alpha^2+4(1-2j-a)\alpha +4j+2a-2b-4n-4.\\
\end{eqnarray*}

Since we have a quadratic equation, the minimum of $g(\alpha)$ is achieved with
$\alpha=\frac{a+2j-1}{2(a+b+2n)}$. Thus
\[g(\alpha) \geq - \frac{(a+2j-1)^2}{a+b+2n}+4j+2a-2b-4n-4.\]

We use the bound on $g(\alpha)$ in order to obtain
\begin{eqnarray*}
\frac{2}{\pi^2}m_2(P)&\geq& \sum_{j=1}^n\left(-
\frac{(a+2j-1)^2}{a+b+2n}+4j+2a-2b-4n-4 \right) \\
&& +2n(n+1) +\frac{n(2n-4a+5b)}{3}+\frac{a+b-ab}{6}\\
&=&\sum_{j=1}^n\left(-
\frac{4j^2}{a+b+2n}+\frac{4(b+2n+1)j}{a+b+2n}-\frac{(a-1)^2}{a+b+2n}\right)\\
&& -\frac{n(4n-2a+b+6)}{3}+\frac{a+b-ab}{6}\\
&=&-
\frac{2n(n+1)(2n+1)}{3(a+b+2n)}+\frac{2(b+2n+1)n(n+1)}{a+b+2n}-\frac{(a-1)^2n}{
a+b+2n}\\
&& -\frac{n(4n-2a+b+6)}{3}+\frac{a+b-ab}{6}\\
&=&\frac{2n(n+1)(3b+4n+2)}{3(a+b+2n)}-\frac{(a-1)^2n}{a+b+2n}\\
&& -\frac{n(4n-2a+b+6)}{3}+\frac{a+b-ab}{6}\\
&=& \frac{1}{6},
\end{eqnarray*}
where the last equality is valid for any of the four cases with $a,b \in
\{0,1\}$.

Thus, 
\begin{eqnarray*}
m_2(P) &\geq& \frac{\pi^2}{12}\cong 0.8224670334\dots.
\end{eqnarray*}
\end{proof}

\begin{rem}\label{remarkreci}
Observe that the previous proof only uses the fact that $P$ is reciprocal with
roots on the unit circle.  Therefore, Theorem \ref{product-cyclotomic} applies
to this family of polynomials.
\end{rem}

In order to prove Theorem \ref{maintheorem1}, we extend Theorem
\ref{product-cyclotomic} to reciprocal polynomials:
\begin{thm}\label{reciprocal}If $P(x) \in \Z[x]$ is reciprocal, then
\[m_{2}(P) \geq \frac{\pi^2}{12}.\]
\end{thm}

We will need the following result which is Remark 9 in \cite{KLO}:
\begin{lem} \label{Remark9-KLO}
For $a,b \in \C,$
\[m(1-a x, 1-b x) = \begin{cases}
\frac12\re \Li_2(a\overline{b}) &\text{ if }|a|,|b|\leq 1,\\\\
\frac12\re \Li_2(b/\overline{a}) &\text{ if }|a|\geq 1,|b|\leq 1,\\\\
\frac12\re \Li_2(1/\overline{a}b)+\log|a|\log|b| &\text{ if }|a|,|b|\geq 1,
\end{cases}\]
where $\Li_2$ is the dilogarithm function.
\end{lem}

\begin{prop}\label{cosinesum}
Let $\tau_1,\dots,\tau_M$ be fixed real numbers in $[0,1)$  and $c_1, \dots,
c_M>0$. The function
\[f(y_1,\dots,y_M)= \sum_{j=1}^M c_i \sum_{n=1}^\infty \frac{y_j^n \cos(2 \pi n
\tau_i)}{n^2}\]
attains its minimum in $[0,1]^M$ at a point where $y_i \in \{0,1\}$ for each
$i$.
\end{prop}

\begin{proof}
For a fixed $\tau \in [0,1),$ we first study the function
\[g(y)= \sum_{n=1}^\infty \frac{y^n \cos(2 \pi n \tau)}{n^2}\] in the interval
$[0,1].$  In this interval, $g(y)$ attains its minimum either at the end points
or when $g'(y) = 0.$  However,
\[g'(y)=\frac{1}{y}\sum_{n=1}^\infty \frac{y^n \cos(2 \pi n \tau)}{n}=
-\frac{1}{y} \log \left|1-y e^{2 \pi i \tau} \right|.\]
Thus, we get a critical point when $\left|1-ye^{2 \pi i \tau} \right|=1$, that
is, when
\[(1-y \cos(2 \pi \tau))^2+ (y \sin(2 \pi \tau))^2=1\]
and therefore, $y_0=2 \cos(2 \pi \tau)$.  We need to determine what kind of
point $y_0$ is.  Observe that
\begin{eqnarray*}
g''(y)&=& 
\frac{1}{y^2}\log \left|1-ye^{2 \pi i \tau} \right|+\frac{1}{y^2} \re
\left(\frac{y e^{2 \pi i \tau}}{1-y e^{2 \pi i \tau}} \right).
\end{eqnarray*}

Thus,
\[g''(y_0)=\frac{1}{y_0^2} \re \left(y_0 e^{2 \pi i \tau}(1-y_0 e^{-2 \pi i
\tau})\right)= \frac{1}{y_0^2} \left(y_0 \cos(2 \pi
\tau)-y_0^2\right)=-\frac{1}{2}<0.\]

Then, $y_0$ is a (local) maximum point for $g(y)$. Therefore, the minimum for
$g(y)$ in $[0,1]$ is either at $y=0$ or $y=1$.  Since each $c_i>0,$ we conclude
that in the interval $[0,1]^M,\,f(y_1,\dots,y_M)$ attains its minimum at a point
where each $y_i$ is either 0 or 1.
\end{proof}

\begin{rem} \label{openinterval}
From the above analysis it also follows that if $f(y_1,\dots,y_M) \geq
f(a_1,\dots,a_M)$ for all $(y_1,\dots,y_M) \in (0,1)^M,$ then each $a_i \in
\{0,1\}.$
\end{rem}

\begin{proof}[Theorem \ref{reciprocal}]. Let $P(x)$ be a reciprocal polynomial
in $\Z[x].$  If $P$ is not monic, we can write $P(x)=CQ(x)$ with $C\in \Z$ and
\[m_2(P)=\log^2 C + 2 \log C m(Q) + m_2(Q)\geq m_2(Q),\]
where we are using that $C\geq 1$ in the inequality. 

Therefore we can assume that $P$ is monic. Thus we write
\[P(x)= (x-1)^a(x+1)^b \prod_{j=1}^J (x-r_j)(x-r_j^{-1})\]
where $|r_j|\leq 1$. We write $r_j=\rho_j e^{2\pi i \mu_j}$ with $0\leq \rho_j
\leq 1$. Here we need to clarify what for $\rho_j=0$ the product
$(x-r_j)(x-r_j^{-1})$ should be interpreted as the product $x \cdot 1$. It is
important to understand that Lemma \ref{Remark9-KLO} is still valid in these
cases since $\Li_2(0)=0$.

In addition, $a,b \in \{0,1\}$ as they account for the fact that we may have an
odd number of factors $x-1$ and/or $x+1$ in the product. 

Then
\begin{eqnarray*}
m_2(P) & = & am_2(x-1)+bm_2(x+1)+2abm(x+1,x-1)+2a \sum_{j=1}^J
m(x-1,x-r_j)\\&&+2b \sum_{j=1}^J m\left(x-1,x-r_j^{-1}\right)+\sum_{1\leq
j_1,j_2\leq J} m(x-r_{j_1},x-r_{j_2})\\
&&+ \sum_{1\leq j_1,j_2\leq J} m\left(x-r_{j_1}^{-1},x-r_{j_2}^{-1}\right)
+2 \sum_{1\leq j_1,j_2\leq J} m\left(x-r_{j_1},x-r_{j_2}^{-1}\right).
\end{eqnarray*}
Using that $m(x+1,x-1)=-\frac{\pi^2}{24}$ and applying Lemma \ref{Remark9-KLO}, 
\begin{eqnarray*}
m_2(P) & = & (a+b-ab)\frac{\pi^2}{12}+(a+b) \sum_{j=1}^J \re \Li_2(r_j) \\
&&+ \sum_{1\leq j_1,j_2\leq J} \left(2\re \Li_2(r_{j_1}\overline{r_{j_2}}) +
\log|r_{j_1}|\log|r_{j_2}|\right).\\
\end{eqnarray*}

By writing the dilogarithm in terms of its power series, we get
\begin{eqnarray}\label{eq:m_2versusm}
m_2(P) & = & (a+b-ab)\frac{\pi^2}{12}+m(P)^2+(a+b) \sum_{j=1}^J
\sum_{n=1}^\infty \frac{\rho_{j}^n\cos(2\pi n\mu_{j})}{n^2} \nonumber \\
&&+ 2\sum_{1\leq j_1,j_2\leq J} \sum_{n=1}^\infty \frac{\rho_{j_1}^n\rho_{j_2}^n
\cos(2\pi n(\mu_{j_1}-\mu_{j_2}))}{n^2}.
\end{eqnarray}

Thus, the problem of minimizing $m_2(P)$ reduces to the problem of minimizing
the terms in the above expression. First let us fix the arguments $\mu_j$. As a
consequence of Proposition \ref{cosinesum} and Remark \ref{openinterval}, we see
that the last term involving a series reaches its minimum when $\rho_j \in
\{0,1\}$. This condition also minimizes the other term involving a series,
although that term can be ignored if $a=b=0$.  This means that there are no
roots with absolute value greater than 1. This also minimizes the first term
$m(P)^2$ which is nonnegative for $P(x)$ monic and zero if all the roots of
$P(x)$ are of absolute value (less than or) equal to 1.  Now if we allow the
arguments $\mu_j$ to vary, the minimum $m_2(P)$ is still attained when all the
roots have absolute value in $\{0,1\}$. Since $P$ is reciprocal and its roots
have absolute value 1, we can apply Theorem \ref{product-cyclotomic} and Remark
\ref{remarkreci} to conclude that 
\[m_2(P) \geq \frac{\pi^2}{12}.\]
\end{proof}

In order to prove Theorem \ref{maintheorem1} we need to say what happens when
$P$ is not reciprocal. 
\begin{lem} \label{simplenonrec} If $P(x) \in \Z[x]$ is nonreciprocal, then
\[m_{2}(P) \geq \frac{\pi^2}{48}.\]
\end{lem}
\begin{proof}
Let $d=\deg P$, and consider $Q(x)=x^dP(x^{-1})$. Thus $P(x)Q(x)\in \Z[x]$ is
reciprocal. Moreover, $m_2(P)=m_2(Q)=m(P,Q)$, thus,
\[m_2(PQ)=m_2(P)+2m(P,Q)+m_2(Q)=4m_2(P).\]
We obtain the desired bound by applying Theorem  \ref{reciprocal} to $PQ$.
\end{proof}

\begin{rem}
While the inequality in Theorem  \ref{reciprocal} is sharp (as
$m_2(x-1)=\frac{\pi^2}{12}$), we do not know what happens with the inequality in
Lemma \ref{simplenonrec}. The best polynomial we were able to find is 
\[m_2(x^3+x+1)\cong 0.3275495729\dots,\]
while $\frac{\pi^2}{48}\cong 0.2056167583\dots.$
\end{rem}

We will use the bound for $m_2(P)$ in order to find a bound for $m_{2h}(P)$.
\begin{prop}\label{even-higher-Mahler}
For any nonzero polynomial $P(x) \in \C[x],$
\begin{enumerate}
\item \[m_{2h}(P) \geq m_2(P)^h,\]
\item \[m_{2h}(P)\geq m(P)^{2h}.\]
\end{enumerate}
\end{prop}

\begin{proof} Part {\em 1}. For any positive integer $h,$ let $f$ and $g$ be
functions such that
\[\frac{1}{2\pi i}\int_{|x|=1}|f|^h \frac{dx}{x} < \infty \qquad \mbox{and}
\qquad \frac{1}{2\pi i}\int_{|x|=1}|g|^{h/(h-1)} \frac{dx}{x} < \infty.\]
Then, by H\"older's inequality, we get that
\begin{equation}\label{eq:holder}
\left(\frac{1}{2\pi i}\int_{|x|=1}|fg|\frac{dx}{x}\right)^h \leq
\left(\frac{1}{2\pi i}\int_{|x|=1}|f|^h \frac{dx}{x}\right)\left(\frac{1}{2\pi
i}\int_{|x|=1}|g|^{h/(h-1)} \frac{dx}{x}\right)^{h-1}.\end{equation}
In particular, taking $f(x) = \log^2|P(x)|$ and $g(x) = 1,$ we get that 
\[m_2(P)^h \leq m_{2h}(P).\]
Part {\em 2}. On the other hand, by taking $f(x) = \log|P(x)|$ and $g(x) = 1,$
and taking $2h$ instead of $h$ in \eqref{eq:holder} we get that 
\[m(P)^{2h} \leq m_{2h}(P).\]
\end{proof}

\begin{proof} [Theorem \ref{maintheorem1}]. By combining Theorem
\ref{reciprocal}, Lemma \ref{simplenonrec}, and Part {\em 1} of Proposition
\ref{even-higher-Mahler} we obtain that, 
\[m_{2h}(P) \geq m_2(P)^{h}\geq 
\begin{cases}
\left(\frac{\pi^2}{12}\right)^h,&\text{ if } P(x)\text{ is reciprocal,}\\
\left(\frac{\pi^2}{48} \right)^h,&\text{ if } P(x)\text{ is non-reciprocal}. 
\end{cases}\]
\end{proof}

\begin{rem} \label{holder}
By Part {\em 2} of Proposition \ref{even-higher-Mahler}, if we assume that the
lowest positive value of $m(P)$ is for Lehmer's degree 10 polynomial, then for
any $P(x) \in \Z[x]$ with $m(P) >0,$
$$m_2(P) \geq (0.1623576120\dots)^2 \cong 0.0263599941\dots.$$
 However, Theorem \ref{maintheorem1} provides us with an unconditional and
stronger lower bound $0.2056167583\dots$ for $m_2(P)$ (and $m_{2h}(P)$).

Analogously, we can use the result of Smyth to find a different bound in Lemma
\ref{simplenonrec}. Smyth \cite{S1} proved that for  $P \in \Z[x]$
nonreciprocal,
\begin{equation*} 
m(P) \geq m(x^3-x-1) \cong 0.2811995743\dots.
\end{equation*}
This can be combined with Part {\em 2} of Proposition \ref{even-higher-Mahler}
to obtain
\[m_2(P) \geq (0.2811995743\dots)^2 \cong 0.0790732005\dots,\]
but this bound is less than $\frac{\pi^2}{48}$ and therefore weaker.
\end{rem}

\section{Explicit formulae for 2-Mahler measures of cyclotomic
polynomials}\label{sec:m2-cyclo-explicit}

While the classical Mahler measure for products of cyclotomic polynomials is
uninteresting, we have seen that the same is not true for  higher Mahler
measures.  In this section we show how to  evaluate $m_2(P)$ for such
polynomials.  We notice that any product of cyclotomic polynomials can be
written as
\[P(x)=\prod_{i=1}^N (x^{d_i}-1)^{e_i},\]
where we allow negative exponents. Therefore, we can compute $m_2(P)$ if we
understand $m(x^a-1,x^b-1)$.

We start by proving the following useful result, which is also of independent
interest.
\begin{prop}\label{S-ab}
For any two positive coprime integers $a$ and $b,$
\[S(a,b):=\sum_{j=0}^{a-1} \sum_{k=0}^{b-1}\left|\frac{k}{b}-\frac{j}{a} \right|
= \frac{2a^2b^2-3ab+a^2+b^2-1}{6ab}.\]
\end{prop}

\begin{proof} First we observe that the term inside the absolute value is
positive when $\frac{j}{a}< \frac{k}{b}$. For fixed $j$, this happens for
$k=b-1, \dots, \left\lfloor\frac{bj}{a}\right \rfloor+1$, that is, for
$b-\left\lfloor\frac{bj}{a}\right \rfloor-1$ values of $k$.  On the other hand,
it is negative when $k=0,\dots, \left\lfloor\frac{bj}{a}\right \rfloor$, that
is, for $\left\lfloor\frac{bj}{a}\right \rfloor+1$ values of $k$. Thus,
$\frac{j}{a}$ appears with negative sign for $b-\left\lfloor\frac{bj}{a}\right
\rfloor-1$ values of $k$ and with positive sign for
$\left\lfloor\frac{bj}{a}\right \rfloor+1$ values of $k$. Putting this into the
equation (together with the same analysis for $k$), we obtain,
\begin{eqnarray*}
S(a,b)&=&\sum_{j=0}^{a-1} \frac{j}{a}\left(2\left\lfloor\frac{bj}{a}\right
\rfloor+ 2-b \right)+\sum_{k=0}^{b-1}
\frac{k}{b}\left(2\left\lfloor\frac{ak}{b}\right \rfloor+ 2-a \right)\\
&=&\frac{(2- b)(a-1)}{2}+\frac{(2-a)(b-1)}{2}+\frac{2}{a}\sum_{j=0}^{a-1}
j\left\lfloor\frac{bj}{a}\right
\rfloor+\frac{2}{b}\sum_{k=0}^{b-1}k\left\lfloor\frac{ak}{b}\right \rfloor.
\end{eqnarray*}

Assume without loss of generality that $a>b$. Let $j_l$ be such that
$\left\lfloor \frac{bj_l}{a}\right\rfloor=l$ and  $\left\lfloor
\frac{b(j_l-1)}{a}\right\rfloor=l-1$. Then $0=j_0<j_1<\dots<j_{b-1}<j_b=a$.
Thus,
\begin{eqnarray*}
\frac{2}{a}\sum_{j=0}^{a-1} j\left\lfloor\frac{bj}{a}\right
\rfloor+\frac{2}{b}\sum_{k=0}^{b-1}k\left\lfloor\frac{ak}{b}\right
\rfloor&=&\frac{2}{a}\sum_{l=1}^b\sum_{j=j_{l-1}}^{
j_l-1}j\left\lfloor\frac{bj}{a}\right
\rfloor+\frac{2}{b}\sum_{k=0}^{b-1}k\left\lfloor\frac{ak}{b}\right \rfloor\\
&=&\frac{2}{a}\sum_{l=1}^b(l-1)\sum_{j=j_{l-1}}^{
j_l-1}j+\frac{2}{b}\sum_{k=0}^{b-1}k\left\lfloor\frac{ak}{b}\right \rfloor\\
&=&\frac{2}{a}\sum_{l=1}^b(l-1)\left(\binom{j_l}{2}-\binom{j_{l-1}}{2}
\right)+\frac{2}{b}\sum_{k=0}^{b-1}k\left\lfloor\frac{ak}{b}\right \rfloor\\
&=&\frac{2}{a}\left( (b-1)\binom{a}{2}-\sum_{l=1}^{b-1} \binom{j_{l}}{2}
\right)+\frac{2}{b}\sum_{k=0}^{b-1}k\left\lfloor\frac{ak}{b}\right \rfloor.\\
\end{eqnarray*}
Notice that for $l>0$,
\[ \frac{b(j_l-1)}{a}<l < \frac{bj_l}{a},\]
which implies that $j_l=\left\lfloor \frac{al}{b}\right\rfloor+1$. Thus, the
above computation equals
\begin{eqnarray*}
&&(a-1)(b-1)-\frac{1}{a}\sum_{l=0}^{b-1} \left(\left\lfloor
\frac{al}{b}\right\rfloor^2 +\left\lfloor
\frac{al}{b}\right\rfloor\right)+\frac{2}{b}\sum_{k=0}^{b-1}k\left\lfloor\frac{
ak}{b}\right \rfloor.\\
\end{eqnarray*}
Because of $\sum_{l=0}^{b-1} \left\lfloor
\frac{al}{b}\right\rfloor=\frac{(a-1)(b-1)}{2}$, the above equals
\begin{eqnarray*}
&&(a-1)(b-1)-\frac{(a-1)(b-1)}{2a}-\frac{1}{a}\sum_{k=0}^{b-1} \left\lfloor
\frac{ak}{b}\right\rfloor^2
+\frac{2}{b}\sum_{k=0}^{b-1}k\left\lfloor\frac{ak}{b}\right \rfloor\\
&=&\frac{(a-1)(2a-1)(b-1)}{2a}-\frac{1}{a}\sum_{k=0}^{b-1} \left(\left\lfloor
\frac{ak}{b}\right\rfloor -\frac{ak}{b}\right)^2
+\frac{1}{a}\sum_{k=0}^{b-1}\frac{a^2k^2}{b^2}.
\end{eqnarray*}

Observe that because $(a,b) = 1,$  as the term $k$ runs through all the residues
modulo $b,$ so does the term $b\left(\frac{ak}{b}-\left\lfloor
\frac{ak}{b}\right\rfloor\right).$  Thus, the above expression is equal to 
\begin{eqnarray*}
&&\frac{(a-1)(2a-1)(b-1)}{2a}-\frac{1}{ab^2}\sum_{k=0}^{b-1} k^2
+\frac{a}{b^2}\sum_{k=0}^{b-1}k^2\\
&=&\frac{(a-1)(2a-1)(b-1)}{2a}+\frac{(a^2-1)(b-1)(2b-1)}{6ab}.
\end{eqnarray*}

Finally, 
\begin{eqnarray*}
S(a,b)&=&\frac{(a-1)+(b-1)-2(a-1)(b-1)}{2}+\frac{(a-1)(b-1)(8ab-a-b-1)}{6ab}\\
\end{eqnarray*}
\end{proof}

From the above proposition, we deduce the following theorem:
\begin{thm}\label{theorem-ab}
For any two positive integers $a$ and $b,$
\[m(x^a-1,x^b-1) =\frac{\pi^2}{12}\frac{(a,b)^2}{ab}.\]
\end{thm}

\begin{proof} First assume that $(a,b)=1$. Applying Lemma
\ref{lem:m2-cyclotomic}, we observe that
\begin{eqnarray*}
m(x^a-1,x^b-1)
&=&\frac{\pi^2}{2}\sum_{j=0}^{a-1} \sum_{k=0}^{b-1}
\left(\left|\frac{k}{b}-\frac{j}{a} \right|^2-\left|\frac{k}{b}-\frac{j}{a}
\right|+\frac16\right) \\
&=&\frac{\pi^2}{2}\sum_{j=0}^{a-1} \sum_{k=0}^{b-1}
\left(\frac{k^2}{b^2}+\frac{j^2}{a^2}-\frac{2jk}{ab}+\frac{1}{6}
-\left|\frac{k}{b}-\frac{j}{a} \right|\right).\\
\end{eqnarray*}

Therefore, applying Proposition \ref{S-ab}, we have 
\begin{eqnarray*}
m(x^a-1,x^b-1)&=&\frac{\pi^2}{2}\left(\frac{2a^2b^2+a^2+b^2-3ab}{6ab} -
\frac{2a^2b^2-3ab+a^2+b^2-1}{6ab} \right)\\
&=&\frac{\pi^2}{12ab}.
\end{eqnarray*}
For general $a$ and $b$ it suffices to notice that the change of variables
$y=x^{(a,b)}$ will not affect the Mahler measure, and thus
\[m(x^a-1,x^b-1)=m\left(x^\frac{a}{(a,b)}-1,x^\frac{b}{(a,b)}-1\right).\]
\end{proof}

From Theorem \ref{theorem-ab}, we deduce the following proposition.
\begin{prop}\label{prop-m-n}
For a positive integer $n,$ let $\phi_n(x)$ denote the $n$-th cyclotomic
polynomial and $\mu$ be the M\"obius function.  For any two positive integers
$m$ and $n,$ 
\[m(\phi_m(x),\phi_n(x))=\frac{\pi^2}{12}\sum_{d_1|m
,d_2|n}\mu\left(\frac{m}{d_1}\right)\mu\left(\frac{n}{d_2}\right)
\frac{(d_1,d_2)^2}{d_1d_2}.\]
\end{prop}

\begin{proof}
We recall that for any positive integer $n,$ 
\[x^n - 1 = \prod_{d\mid n}\phi_d(x).\]
Thus, by the multiplicative M\"obius inversion formula, we get that
\[\phi_n(x) = \prod_{d\mid n}(x^d-1)^{\mu(n/d)}.\]
From the above and from Theorem \ref{theorem-ab}, 
\begin{eqnarray*}
m(\phi_m(x),\phi_n(x))&=& \sum_{d_1\mid m ,d_2\mid
n}\mu\left(\frac{m}{d_1}\right)\mu\left(\frac{n}{d_2}\right)m(x^{d_1}-1,x^{d_2}
-1)\\
&=&\frac{\pi^2}{12}\sum_{d_1\mid m ,d_2\mid
n}\mu\left(\frac{m}{d_1}\right)\mu\left(\frac{n}{d_2}\right)
\frac{(d_1,d_2)^2}{d_1d_2}.
\end{eqnarray*}
\end{proof}

\begin{prop}\label{prop-A} Let $p$ be a positive prime. We have the following
transformations
\begin{enumerate}
\item For $k\geq l \geq 1$ and $p \nmid mn$,
\[m(\phi_{p^km}(x),\phi_{p^ln}(x))=\frac{2}{p^{k-l}}\left(1-\frac{1}{p}
\right)m(\phi_{m}(x),\phi_{n}(x)).\]
\item For $k\geq 1$ and $p\nmid mn$,
\[m(\phi_{p^km}(x),\phi_{n}(x))=-\frac{1}{p^{k-1}}\left(1-\frac{1}{p}
\right)m(\phi_{m}(x),\phi_{n}(x)).\]
\end{enumerate}
\end{prop}

\begin{proof} Part {\em 1.} 
Using Proposition \ref{prop-m-n}, we have 
\[m(\phi_{p^k m}(x),\phi_{p^l n}(x))=\frac{\pi^2}{12}\sum_{d_1\mid p^km ,d_2\mid
p^ln}\mu\left(\frac{p^km}{d_1}\right)\mu\left(\frac{p^ln}{d_2}\right)
\frac{(d_1,d_2)^2}{d_1d_2}.\]

It is clear that only the terms with $p^{k-1}\mid d_1$ and $p^{l-1}\mid d_2$ are
nonzero, since otherwise the M\"obius function factors yield zero.  We write
$d_1=p^{k-1}e_1$ and $d_2=p^{l-1}e_2$. Thus

\[m(\phi_{p^k m}(x),\phi_{p^l n}(x))=\frac{\pi^2}{12}\sum_{e_1\mid pm ,e_2\mid
pn}\mu\left(\frac{pm}{e_1}\right)\mu\left(\frac{pn}{e_2}\right)
\frac{(e_1,e_2)^2}{e_1e_2p^{k-l}}.\]
If $p$ divides both $e_1$ and $e_2$ or $p$ does not divide either of them, we
get terms of the form 
\[\mu\left(\frac{m}{f_1}\right)\mu\left(\frac{n}{f_2}\right)
\frac{(f_1,f_2)^2}{f_1f_2p^{k-l}}\]
 with $p\nmid f_i$. 
 
 If, on the other hand, $p$ divides exactly one of the $e_i$'s, we get terms of
the form 
\[-\mu\left(\frac{m}{f_1}\right)\mu\left(\frac{n}{f_2}\right)
\frac{(f_1,f_2)^2}{f_1f_2p^{k-l+1}}\]
 with $p\nmid f_i$. 

Putting all of these ideas together, we obtain
\begin{eqnarray*}
m(\phi_{p^k m}(x),\phi_{p^l
n}(x))&=&\frac{\pi^2}{12}\frac{2}{p^{k-l}}\left(1-\frac{1}{p} \right)\sum_{f_1|m
,f_2|n}\mu\left(\frac{m}{f_1}\right)\mu\left(\frac{n}{f_2}\right)
\frac{(f_1,f_2)^2}{f_1f_2}\\
&=&\frac{2}{p^{k-l}}\left(1-\frac{1}{p} \right)m(\phi_{ m}(x),\phi_{n}(x)),
\end{eqnarray*}
which proves the first part of the proposition.

Part {\em 2.} Once again, by Proposition \ref{prop-m-n}, we can write
\[m(\phi_{p^k m}(x),\phi_{n}(x))=\frac{\pi^2}{12}\sum_{d_1\mid p^km ,d_2\mid
n}\mu\left(\frac{p^km}{d_1}\right)\mu\left(\frac{n}{d_2}\right)
\frac{(d_1,d_2)^2}{d_1d_2}.\]
As before, it is clear that $p^{k-1}\mid d_1$ in the nonzero terms, and we can
write $d_1=p^{k-1}e_1$.

If $p\mid e_1$, we obtain 
\[\mu\left(\frac{m}{f_1}\right)\mu\left(\frac{n}{d_2}\right)
\frac{(f_1,d_2)^2}{f_1d_2p^{k}}\]
 with $p\nmid f_1$. 
 
 If $p\nmid e_1$, we obtain
\[-\mu\left(\frac{m}{e_1}\right)\mu\left(\frac{n}{d_2}\right)
\frac{(e_1,d_2)^2}{e_1d_2p^{k-1}}.\]
Thus,
\begin{eqnarray*}
m(\phi_{p^k
m}(x),\phi_{n}(x))&=&-\frac{\pi^2}{12}\frac{1}{p^{k-1}}\left(1-\frac{1}{p}
\right)\sum_{e_1\mid m ,d_2\mid
n}\mu\left(\frac{m}{e_1}\right)\mu\left(\frac{n}{d_2}\right)
\frac{(e_1,d_2)^2}{e_1d_2}\\
&=&-\frac{1}{p^{k-1}}\left(1-\frac{1}{p} \right)m(\phi_{m}(x),\phi_{n}(x)),
\end{eqnarray*}
proving the second part of the proposition.
\end{proof}

\begin{proof} [Theorem \ref{m2-cyclotomic}] We write the prime factorizations of
$m$ and $n$ as $m=p_1^{k_1}\dots p_r^{k_r} q_1^{h_1} \dots q_u^{h_u}$ and $n=
p_1^{l_1}\dots p_r^{l_r}t_1^{j_1}\dots t_v^{j_v}$, where all the exponents are
positive integers and the primes $q$'s are different from the primes $t$'s. Thus
$r=r((m,n))$. By applying Proposition \ref{prop-A}, we obtain
\begin{eqnarray*}
&&m(\phi_{m}(x),\phi_{n}(x))\\
&=&2^r\prod_{i=1}^r\left(
\frac{1}{p_i^{|k_i-l_i|}}\left(1-\frac{1}{p_i}\right)\right)(-1)^{u+v}\prod_{i=1
}^u \left( \frac{1}{q_i^{h_i-1}}\left(1-\frac{1}{q_i}\right)\right)\prod_{i=1}^v
\left( \frac{1}{t_i^{j_i-1}}\left(1-\frac{1}{t_i}\right)\right) m_2(x-1)\\
&=&\frac{\pi^2}{12} 2^r(-1)^{u+v}\prod_{i=1}^r
\frac{p_i^{\min\{k_i,l_i\}}}{p_i^{\max\{k_i,l_i\}}}\prod_{i=1}^u 
\frac{q_i}{q_i^{h_i}}\prod_{i=1}^v  \frac{t_i}{t_i^{j_i}}\prod_{p\mid
mn}\left(1-\frac{1}{p}\right)\\
&=&\frac{\pi^2}{12} 2^{r((m,n))}(-1)^{r(m)+r(n)}\frac{(m,n)}{[m,n]}
\prod_{i=1}^u  q_i\prod_{i=1}^v t_i\prod_{p\mid mn}\left(1-\frac{1}{p}\right)\\
&=&\frac{\pi^2}{12}\frac{2^{r((m,n))}(-1)^{r(m)+r(n)}(m,n)}{[m,n]}
\left(\prod_{p\mid mn, p \nmid (m,n)} p\right) \frac{\varphi([m,n])}{[m,n]}.
\end{eqnarray*}

\end{proof}

\section{Explicit formulae for 3-Mahler measures of some particular polynomials}
\label{sec:m3}
In this section, we address the case of $m_3(P)$ for $P$ a product of cyclotomic
polynomials. Our starting point is Remark 10 from \cite{KLO}, which is the
following statement:
\begin{prop}\label{RemarkKLO} 
We have
\begin{eqnarray*}
m(1-x,1-e^{2\pi i \alpha} x, 1-e^{2 \pi i \beta} x) 
&=& - \frac14 \sum_{1 \leq k,l} \frac{\cos 2\pi((k+l)\beta-l \alpha)}{k l (k+l)}
\\
&& 
- \frac14 \sum_{1 \leq k,m } \frac{\cos 2\pi((k+m)\alpha-m \beta)}{k m (k+m)} \\
&& - \frac14 \sum_{1 \leq l,m } \frac{\cos 2\pi(l \alpha+m\beta )}{l m(l+m)}. 
\end{eqnarray*}
\end{prop}

First, we express the above formula in terms of Clausen functions.
\begin{prop}\label{prop-m3}For $0\leq \alpha,\beta < 1$, we have
\begin{eqnarray*}
2m(1-x,1-e^{2\pi i \alpha} x, 1-e^{2 \pi i \beta}
x)&=&S_1(2\pi(\beta-\alpha))S_2(2 \pi \beta)+S_2(2\pi(\beta-\alpha))S_1(2 \pi
\beta)\\
&&- C_3(2 \pi \alpha)+S_1(2\pi(\alpha-\beta))S_2(2 \pi \alpha) \\
&&+S_2(2\pi(\alpha-\beta))S_1(2 \pi \alpha)- C_3(2 \pi \beta)\\
&&+S_1(2\pi \alpha)S_2(2 \pi \beta) +S_2(2\pi \alpha)S_1(2 \pi \beta)\\
&&- C_3(2 \pi (\beta-\alpha)).\\
 \end{eqnarray*}
\end{prop}

\begin{proof}
Our starting point will be the following elementary identity
\begin{eqnarray*}
-\frac{1}{l m (l+m)}&=&
\left(\frac{1}{l (l+m)^2}+ \frac{1}{l^2 (l+m)}\right) + \left(\frac{1}{m
(l+m)^2} + \frac{1}{m^2 (l+m)}\right)\\
&&-\left(\frac{1}{l^2 m}+\frac{1}{l m^2}\right).
\end{eqnarray*}

Notice that
\begin{eqnarray*}
&&\sum_{1 \leq l,m } \frac{\cos 2\pi(l \alpha+m\beta )}{l (l+m)^2}+ \sum_{1 \leq
l,m} \frac{\cos 2\pi((l+m) \alpha-m\beta )}{l^2 (l+m)}\\
&&= \sum_{1 \leq l,m } \frac{\cos 2\pi(l(\alpha-\beta)+(l+m)\beta )}{l (l+m)^2}
+ \sum_{1 \leq l,m} \frac{\cos 2\pi((l+m)(\alpha-\beta)+l\beta )}{l^2 (l+m)}\\
&&=\sum_{1 \leq l,k } \frac{\cos 2\pi(l(\alpha-\beta)+k\beta )}{l k^2}-\sum_{1
\leq k} \frac{\cos 2\pi(k\alpha)}{k^3}.
\end{eqnarray*}
Using the fact that $\cos 2\pi(l(\alpha-\beta)+k\beta )= \cos  (2\pi
l(\alpha-\beta)) \cos (2 \pi k\beta) -  \sin  (2\pi l(\alpha-\beta)) \sin (2 \pi
k\beta)$, we can rewrite the previous identity as
\begin{eqnarray*}
&&\sum_{1 \leq l} \frac{\cos (2\pi l(\alpha-\beta))}{l}  \sum_{1 \leq k}
\frac{\cos (2\pi k\beta)}{k^2}-\sum_{1 \leq l} \frac{\sin(2\pi
l(\alpha-\beta))}{l}  \sum_{1 \leq k} \frac{\sin (2\pi k\beta)}{k^2} \\
&&-\sum_{1 \leq k} \frac{\cos (2\pi k\alpha)}{k^3} \\
&=&C_1(2\pi(\alpha-\beta))C_2(2 \pi \beta) -S_1(2\pi(\alpha-\beta))S_2(2 \pi
\beta) - C_3(2 \pi \alpha).
\end{eqnarray*}
By exchanging $\beta$ and $\alpha-\beta$ and adding, we obtain
\begin{eqnarray*}
&&\sum_{1 \leq l,m} \frac{\cos 2\pi(l \alpha+m\beta )}{l (l+m)^2}+\sum_{1 \leq
l,m} \frac{\cos 2\pi(l \alpha+m\beta )}{l^2 (l+m)}\\
&&+\sum_{1 \leq l,m } \frac{\cos 2\pi((l+m) \alpha-m\beta )}{l(l+m)^2}+ \sum_{1
\leq l,m} \frac{\cos 2\pi((l+m) \alpha-m\beta )}{l^2 (l+m)} \\
&=& C_1(2\pi(\alpha-\beta))C_2(2 \pi \beta) -S_1(2\pi(\alpha-\beta))S_2(2 \pi
\beta)\\
&&+ C_2(2\pi(\alpha-\beta))C_1(2 \pi \beta) -S_2(2\pi(\alpha-\beta))S_1(2 \pi
\beta)- 2C_3(2 \pi \alpha).
\end{eqnarray*}
Analogously we obtain
\begin{eqnarray*}
&&\sum_{1 \leq l,m} \frac{\cos 2\pi(l \alpha+m\beta )}{m (l+m)^2}+ \sum_{1 \leq
l,m} \frac{\cos 2\pi(l \alpha+m\beta )}{m^2 (l+m)}\\
&&+\sum_{1 \leq l,m} \frac{\cos 2\pi((l+m) \beta-l\alpha )}{m (l+m)^2}+\sum_{1
\leq l,m} \frac{\cos 2\pi((l+m) \beta-l\alpha )}{m^2 (l+m)}\\
&=&C_1(2\pi(\beta-\alpha))C_2(2 \pi \alpha) -S_1(2\pi(\beta-\alpha))S_2(2 \pi
\alpha) \\
&&+C_2(2\pi(\beta-\alpha))C_1(2 \pi \alpha) -S_2(2\pi(\beta-\alpha))S_1(2 \pi
\alpha) - 2C_3(2 \pi \beta),
\end{eqnarray*}
and
\begin{eqnarray*}
&&\sum_{1 \leq l,m} \frac{\cos 2\pi((l+m) \beta-l\alpha )}{l (l+m)^2}+\sum_{1
\leq l,m} \frac{\cos 2\pi((l+m) \beta-l\alpha )}{l^2 (l+m)}\\
&&+\sum_{1 \leq l,m} \frac{\cos 2\pi((l+m) \alpha-m\beta )}{m(l+m)^2}+ \sum_{1
\leq l,m} \frac{\cos 2\pi((l+m) \alpha-m\beta )}{m^2 (l+m)} \\
&=&C_1(2\pi \alpha)C_2(2 \pi \beta) +S_1(2\pi \alpha)S_2(2 \pi \beta) \\
&&+ C_2(2\pi \alpha)C_1(2 \pi \beta) +S_2(2\pi \alpha)S_1(2 \pi \beta)- 2C_3(2
\pi (\beta-\alpha)).
\end{eqnarray*}

On the other hand, we have
\begin{eqnarray*}
&&\sum_{1 \leq l,m} \frac{\cos 2\pi(l \alpha+m\beta )}{l^2 m} + \sum_{1 \leq
l,m} \frac{\cos 2\pi(l \alpha+m\beta )}{l m^2}\\
&&= \sum_{1 \leq l} \frac{\cos (2\pi l \alpha)}{l^2}\sum_{1 \leq m} \frac{\cos
(2\pi m \beta)}{m} -\sum_{1 \leq l} \frac{\sin (2\pi l \alpha)}{l^2}\sum_{1 \leq
m} \frac{\sin (2\pi m \beta)}{m} \\
&&+ \sum_{1 \leq l} \frac{\cos (2\pi l \alpha)}{l}\sum_{1 \leq m} \frac{\cos
(2\pi m \beta)}{m^2} -\sum_{1 \leq l} \frac{\sin (2\pi l \alpha)}{l}\sum_{1 \leq
m} \frac{\sin (2\pi m \beta)}{m^2} \\
&&=C_2(2\pi\alpha) C_1(2\pi\beta) - S_2(2\pi\alpha) S_1(2\pi\beta)
+C_1(2\pi\alpha) C_2(2\pi\beta) - S_1(2\pi\alpha) S_2(2\pi\beta). 
\end{eqnarray*}
As before, we can obtain similar identities by exchanging $\beta$ and
$\alpha-\beta$ and $\alpha$ and $\beta-\alpha$.

By combining the previous results, we obtain the desired formula:
\begin{eqnarray*}
&& - \sum_{1 \leq k,l} \frac{\cos 2\pi((k+l)\beta-l \alpha)}{k l (k+l)} -
\sum_{1 \leq k,m} \frac{\cos 2\pi((k+m)\alpha-m \beta)}{k m (k+m)}\\
&& -\sum_{1 \leq l,m} \frac{\cos 2\pi(l \alpha+m\beta )}{l m(l+m)}\\
&= & 2S_1(2\pi(\beta-\alpha))S_2(2 \pi \beta)+2S_2(2\pi(\beta-\alpha))S_1(2 \pi
\beta)- 2C_3(2 \pi \alpha)\\
&&+2S_1(2\pi(\alpha-\beta))S_2(2 \pi \alpha) +2S_2(2\pi(\alpha-\beta))S_1(2 \pi
\alpha) - 2C_3(2 \pi \beta)\\
&& +2S_1(2\pi \alpha)S_2(2 \pi \beta) +2S_2(2\pi \alpha)S_1(2 \pi \beta)- 2C_3(2
\pi (\beta-\alpha)).\\
\end{eqnarray*}
\end{proof}

We are now ready to prove Theorem \ref{product-cyclotomic-m3}:

\begin{proof} [Theorem \ref{product-cyclotomic-m3}] We express $m_3(P)$ in terms
of the arguments $\alpha_i$:
\begin{eqnarray*}
m_3(P)&=&\sum_{1\leq j,k,l \leq n} m(1-e^{2\pi i \alpha_j}x,1-e^{2\pi i
\alpha_k}x,1-e^{2\pi i \alpha_l}x )\\&=&\sum_{1\leq j,k,l \leq n}
m(1-x,1-e^{2\pi i (\alpha_k-\alpha_j)}x, 1-e^{2\pi i (\alpha_l-\alpha_j)}x ).
\end{eqnarray*}
We now apply Proposition \ref{prop-m3}
\begin{eqnarray*}
 2m_3(P)&=&-\sum_{1\leq j,k,l \leq n} (C_3(2 \pi (\alpha_k-\alpha_j))+ C_3(2 \pi
(\alpha_l-\alpha_k))+C_3(2 \pi (\alpha_j-\alpha_l)))\\
&&+\sum_{1\leq j,k,l \leq n} (S_1(2\pi(\alpha_l-\alpha_k))S_2(2 \pi
(\alpha_l-\alpha_j))+S_2(2\pi(\alpha_l-\alpha_k))S_1(2 \pi
(\alpha_l-\alpha_j))\\
&&+S_1(2\pi(\alpha_k-\alpha_l))S_2(2 \pi (\alpha_k-\alpha_j))
+S_2(2\pi(\alpha_k-\alpha_l))S_1(2 \pi (\alpha_k-\alpha_j))\\
&& +S_1(2\pi (\alpha_k-\alpha_j))S_2(2 \pi (\alpha_l-\alpha_j)) +S_2(2\pi
(\alpha_k-\alpha_j))S_1(2 \pi (\alpha_l-\alpha_j)))\\
&=& -3 n \sum_{1\leq k,l\leq n} C_3(2 \pi (\alpha_l-\alpha_k))\\
&&+3\sum_{1\leq j,k,l \leq n} (S_1(2\pi(\alpha_l-\alpha_k))S_2(2 \pi
(\alpha_l-\alpha_j))+S_2(2\pi(\alpha_l-\alpha_k))S_1(2 \pi
(\alpha_l-\alpha_j)))\\
&=&-3 n \sum_{1\leq k,l \leq n} C_3(2 \pi (\alpha_l-\alpha_k))+6\sum_{1 \leq k,l
\leq n} S_2(2 \pi
(\alpha_l-\alpha_k))\sum_{j=1}^nS_1(2\pi(\alpha_l-\alpha_j)).\\
\end{eqnarray*}

We will use the following formula 
\begin{equation}\label{S_1}
S_1(2\pi \gamma)= \left \{ \begin{array}{ll}-\pi \left(\gamma -\frac12\right) &
0 < \gamma <1, \\-\pi\gamma= 0& \gamma=0,\\-\pi \left(\gamma +\frac12\right) &
-1 < \gamma <0,\end{array}\right. 
\end{equation}
which can be deduced from the fact that $S_1(2\pi \gamma)=\im (-\log (1 - e^{2
\pi i \gamma}))$.
Thus,
\[-\frac23m_3(P)= n \sum_{1\leq k,l \leq n} C_3(2 \pi (\alpha_l-\alpha_k))+2 \pi
\sum_{1 \leq k,l \leq n} S_2(2 \pi (\alpha_l-\alpha_k))
\left(\sum_{j=1}^n(\alpha_l-\alpha_j) -  \frac{l-1}{2}+\frac{n-l}{2}\right).\]
Notice that $\sum_{1 \leq k,l \leq n}  S_2(2 \pi (\alpha_l-\alpha_k))=0$ because
$S_2(2 \pi (\alpha_l-\alpha_k))$ cancels with $S_2(2 \pi (\alpha_k-\alpha_l))$.
Then 
\begin{eqnarray*}
-\frac23m_3(P)&=&n \sum_{1\leq k,l \leq n} C_3(2 \pi (\alpha_l-\alpha_k))+2 \pi
\sum_{1 \leq k,l \leq n} S_2(2 \pi (\alpha_l-\alpha_k)) \left(n\alpha_l
-\sum_{j=1}^n \alpha_j -  \frac{l-1}{2}+\frac{n-l}{2}\right)\\
&=&n \sum_{1\leq k,l \leq n} C_3(2 \pi (\alpha_l-\alpha_k))+2 \pi \sum_{1 \leq
k,l \leq n} S_2(2 \pi (\alpha_l-\alpha_k)) \left(n\alpha_l - l\right).\\
\end{eqnarray*}

By exchanging $k$ with $l$ and taking the semi-sum, we obtain
\begin{eqnarray*}
 -\frac23m_3(P)&=&n \sum_{1 \leq k,l \leq n} C_3(2 \pi (\alpha_l-\alpha_k))+\pi
\sum_{1 \leq k,l \leq n} S_2(2 \pi (\alpha_l-\alpha_k))
\left(n(\alpha_l-\alpha_k) + k-l\right)\\
&=&n^2 \zeta(3) + 2n \sum_{1\leq k<l\leq n} C_3(2 \pi
(\alpha_l-\alpha_k))\\&&+2\pi \sum_{1\leq k<l\leq n}  S_2(2 \pi
(\alpha_l-\alpha_k)) \left(n(\alpha_l-\alpha_k) -(l-k)\right).
\end{eqnarray*}
\end{proof}

\begin{thm} \label{thm:m3} Let $(a,b,c)=1$. For integers $d,m$, let
$d_m=\frac{d}{(d,m)}$ and $m_d=\frac{m}{(d,m)}$.  Let $n$ be another integer such
that $(d,m)\mid n$.  Then we denote by $[d_m^{-1}n]_{m_d}$ the unique integer
between $0$ and $m_d-1$ such that it is the solution to the equation $d_mx\equiv
n (\mod\, m_d)$.  With this notation we have
\[-2m(x^a-1,x^b-1,x^c-1)\]
\begin{eqnarray*}
&=& abc\left(\frac{1}{[a,b]^3}+\frac{1}{[b,c]^3}+\frac{1}{[c,a]^3}\right)
\zeta(3)\\
&&-\frac{\pi}{2c(a,b)}\sum_{{h=1}\atop{b_a\nmid ch}}^\infty  \frac{\cot
\left(\pi \frac{[a_b^{-1}]_{b_a}ch}{b_a}
\right)}{h^2}-\frac{\pi}{2b(a,c)}\sum_{{h=1}\atop{ c_a\nmid bh}}^\infty
\frac{\cot \left(\pi \frac{[a_c^{-1}]_{c_a}bh}{c_a} \right) }{h^2} \\
&&-\frac{\pi}{2a(b,c)}\sum_{{h=1}\atop{c_b\nmid ah}}^\infty  \frac{\cot
\left(\pi \frac{[b_c^{-1}]_{c_b}ah}{c_b}
\right)}{h^2}-\frac{\pi}{2c(b,a)}\sum_{{h=1}\atop{ a_b\nmid ch}}^\infty
\frac{\cot \left(\pi \frac{[b_a^{-1}]_{a_b}ch}{a_b} \right)}{h^2}\\ 
&&-\frac{\pi}{2b(c,a)}\sum_{{h=1}\atop{a_c\nmid bh}}^\infty  \frac{\cot
\left(\pi \frac{[c_a^{-1}]_{a_c}bh}{a_c}
\right)}{h^2}-\frac{\pi}{2a(c,b)}\sum_{{h=1}\atop{b_c\nmid ah}}^\infty 
\frac{\cot \left(\pi \frac{[c_b^{-1}]_{b_c}ah}{b_c} \right)}{h^2}.\\
\end{eqnarray*}

\end{thm}
\begin{proof} First notice that the assumption that $(a,b,c)=1$ is not
restrictive, since we have that
$m(x^a-1,x^b-1,x^c-1)=m\left(x^\frac{a}{(a,b,c)}-1,x^\frac{b}{(a,b,c)}-1,x^\frac
{c}{(a,b,c)}-1\right)$.

By applying the same ideas as in Lemma \ref{lem:m2-cyclotomic},
\begin{eqnarray*}
2m(x^a-1,x^b-1,x^c-1)&=&2\sum_{j=0}^{a-1} \sum_{k=0}^{b-1} \sum_{l=0}^{c-1}
m(1-e^{2\pi i j/a}x,1-e^{2\pi i k/b}x, 1-e^{2\pi i l/c}x)\\&=& \sum_{j=0}^{a-1}
\sum_{k=0}^{b-1} \sum_{l=0}^{c-1} 2m(1-x,1-e^{2\pi i k/b-2\pi i j/a}x,1-e^{2\pi
i l/c-2\pi i j/a}x)\\
&=&\sum_{j=0}^{a-1} \sum_{k=0}^{b-1} \sum_{l=0}^{c-1} S(j,k,l).
\end{eqnarray*}

By applying Proposition \ref{prop-m3}, we obtain that each term in the sum is
\begin{eqnarray*}
S(j,k,l)&:=& S_2\left(2 \pi
\left(\frac{l}{c}-\frac{j}{a}\right)\right)\left(S_1\left(2\pi\left(\frac{l}{c}
-\frac{k}{b}\right)\right)+ S_1\left(2
\pi\left(\frac{k}{b}-\frac{j}{a}\right)\right)\right) \\
&&+S_2\left(2\pi\left(\frac{k}{b}-\frac{l}{c}\right)\right) \left(S_1\left(2
\pi\left(\frac{k}{b}-\frac{j}{a}\right)\right) +S_1\left(2 \pi
\left(\frac{j}{a}-\frac{l}{c}\right)\right)\right) \\
&& +S_2\left(2 \pi\left(\frac{j}{a}-\frac{k}{b}\right)\right)\left(S_1\left(2
\pi
\left(\frac{j}{a}-\frac{l}{c}\right)\right)+S_1\left(2\pi\left(\frac{l}{c}-\frac
{k}{b}\right)\right) \right)\\
&&- C_3\left(2 \pi \left(\frac{l}{c}-\frac{j}{a}\right)\right)-
C_3\left(2\pi\left(\frac{k}{b}-\frac{l}{c}\right)\right)- C_3\left(2
\pi\left(\frac{j}{a}-\frac{k}{b}\right)\right).\\
\end{eqnarray*}
We will apply formula \eqref{S_1}. First assume that
$\frac{l}{c}>\frac{k}{b}>\frac{j}{a}$. Then
\begin{eqnarray*}
S(j,k,l)&=& -\pi S_2\left(2 \pi
\left(\frac{l}{c}-\frac{j}{a}\right)\right)\left(
\frac{l}{c}-\frac{j}{a}-1\right) \\
&&-\pi S_2\left(2\pi\left(\frac{k}{b}-\frac{l}{c}\right)\right) 
\left(\frac{k}{b}-\frac{l}{c} \right) -\pi S_2\left(2
\pi\left(\frac{j}{a}-\frac{k}{b}\right)\right)\left(\frac{j}{a}-\frac{k}{b}
\right)\\
&&- C_3\left(2 \pi \left(\frac{l}{c}-\frac{j}{a}\right)\right)-
C_3\left(2\pi\left(\frac{k}{b}-\frac{l}{c}\right)\right)- C_3\left(2
\pi\left(\frac{j}{a}-\frac{k}{b}\right)\right).\\
\end{eqnarray*}
Now assume that $\frac{l}{c}>\frac{k}{b}=\frac{j}{a}$ or
$\frac{l}{c}=\frac{k}{b}>\frac{j}{a}$.  Then
\begin{eqnarray*}
S(j,k,l)&=& -2\pi S_2\left(2 \pi
\left(\frac{l}{c}-\frac{j}{a}\right)\right)\left(\frac{l}{c}-\frac{j}{a}-\frac{1
}{2}\right) -2 C_3\left(2 \pi \left(\frac{l}{c}-\frac{j}{a}\right)\right)-
\zeta(3).\\
\end{eqnarray*}
By considering similar analysis for other cases, we finally get
\[-2m(x^a-1,x^b-1,x^c-1)\]
\begin{eqnarray} \label{crazyeq}
&=&  \sum_{k=0}^{b-1} \sum_{l=0}^{c-1} 
\left(aC_3\left(2\pi\left(\frac{k}{b}-\frac{l}{c}\right)\right)+\pi
S_2\left(2\pi\left(\frac{k}{b}-\frac{l}{c}\right)\right) 
\left(a\left(\frac{k}{b}-\frac{l}{c}\right) +
H_a\left(\frac{k}{b},\frac{l}{c}\right) \right)\right)\nonumber \\&&+
\sum_{j=0}^{a-1} \sum_{l=0}^{c-1} \left(bC_3\left(2 \pi
\left(\frac{l}{c}-\frac{j}{a}\right)\right) +\pi S_2\left(2 \pi
\left(\frac{l}{c}-\frac{j}{a}\right)\right)\left(b\left(
\frac{l}{c}-\frac{j}{a}\right)+
H_b\left(\frac{l}{c},\frac{j}{a}\right)\right)\right) \nonumber \\
&&+  \sum_{j=0}^{a-1} \sum_{k=0}^{b-1} \left(cC_3\left(2
\pi\left(\frac{j}{a}-\frac{k}{b}\right)\right)+\pi S_2\left(2
\pi\left(\frac{j}{a}-\frac{k}{b}\right)\right)\left(c\left(\frac{j}{a}-\frac{k}{
b}\right) +H_c\left(\frac{j}{a},\frac{k}{b}\right)\right)\right),\nonumber\\
\end{eqnarray}
where, $H_d\left(\frac{r}{s},\frac{u}{v} \right)$ for $\frac{r}{s}<\frac{u}{v}$
denotes the number of rational numbers of the form $\frac{m}{d}$ with $m\in \Z$
that belong to the interval $\left[\frac{r}{s},\frac{u}{v} \right]$ with the
following conventions: the cases in which $\frac{m}{d}=\frac{r}{s}$ and
$\frac{m}{d}=\frac{u}{v}$ are counted with weight $\frac{1}{2}$ instead of 1,
and $H_d\left(\frac{u}{v}, \frac{r}{s}\right)= -H_d\left(\frac{r}{s},\frac{u}{v}
\right)$.
It is not hard to see that 
\[H_d\left(\frac{r}{s},\frac{u}{v} \right)=\frac{\left\lfloor
\frac{du}{v}\right\rfloor+\left \lceil\frac{du}{v}\right\rceil-\left
\lfloor\frac{dr}{s}\right\rfloor-\left\lceil\frac{dr}{s}\right\rceil}{2}.\]
We will also use the following notation
\[\{\alpha\}_2:= \alpha - \frac{\left\lfloor \alpha\right\rfloor+\left
\lceil\alpha\right\rceil}{2}=\left \{\begin{array}{ll}\alpha-\left\lfloor
\alpha\right\rfloor-\frac{1}{2} & \alpha \not \in \Z,\\ 0 & \alpha \in \Z,  
\end{array}\right.\]
whose Fourier series is
\[\left\{\alpha\right\}_2=-\frac{1}{\pi}\sum_{h=1}^\infty \frac{\sin(2\pi\alpha
h)}{h}.\]

We first study the terms of \eqref{crazyeq} with $C_3$. In this case we get
\[\sum_{k=0}^{b-1}
\sum_{l=0}^{c-1}C_3\left(2\pi\left(\frac{k}{b}-\frac{l}{c}\right)\right)\]
\begin{eqnarray*}
&=& \sum_{n=1}^\infty \sum_{k=0}^{b-1} \sum_{l=0}^{c-1} \frac{\cos\left(2\pi
\left(\frac{k}{b}-\frac{l}{c}\right)n\right)}{n^3}\\
&=&\sum_{n=1}^\infty  \frac{\sum_{k=0}^{b-1}\cos\left(2\pi
\frac{kn}{b}\right)\sum_{l=0}^{c-1} \cos\left(2\pi
\frac{ln}{c}\right)+\sum_{k=0}^{b-1}\sin\left(2\pi
\frac{kn}{b}\right)\sum_{l=0}^{c-1} \sin\left(2\pi \frac{ln}{c}\right)}{n^3}\\
&=&\sum_{{n=1}\atop{b\mid n, c\mid n}}^\infty 
\frac{bc}{n^3}=\frac{bc}{[b,c]^3}\zeta(3).\\
\end{eqnarray*}
Here we have used that $\sum_{k=0}^{b-1}\sin\left(2\pi \frac{kn}{b}\right)=0$
for any $n$, $\sum_{k=0}^{b-1}\cos\left(2\pi \frac{kn}{b}\right)=0$ for $b\nmid
n$ and $\sum_{k=0}^{b-1}\cos\left(2\pi \frac{kn}{b}\right)=b$ for $b\mid n$.

Regarding the terms of \eqref{crazyeq} with $S_2$, we obtain,
\begin{eqnarray*}
&&\sum_{k=0}^{b-1}
\sum_{l=0}^{c-1}S_2\left(2\pi\left(\frac{k}{b}-\frac{l}{c}\right)\right) 
\left(\left\{\frac{ak}{b}\right\}_2-\left\{\frac{al}{c}\right\}_2\right)\\
 &=& \sum_{n=1}^\infty \sum_{k=0}^{b-1}
\sum_{l=0}^{c-1}\frac{\sin\left(2\pi\left(\frac{k}{b}-\frac{l}{c}
\right)n\right)\left(\left\{\frac{ak}{b}\right\}_2-\left\{\frac{al}{c}\right\}
_2\right)}{n^2} \\
\hspace{-.5cm}&=& \hspace{-.2cm}\sum_{n=1}^\infty 
\frac{\sum_{k=0}^{b-1}\sin\left(2\pi\left(\frac{k}{b}\right)n\right)\left\{\frac
{ak}{b}\right\}_2\sum_{l=0}^{c-1}\cos\left(2\pi\left(\frac{l}{c}\right)n\right)
-\sum_{k=0}^{b-1}\sin\left(2\pi\left(\frac{k}{b}\right)n\right)\sum_{l=0}^{c-1}
\cos\left(2\pi\left(\frac{l}{c}\right)n\right)\left\{\frac{al}{c}\right\}_2}{n^2
} \\
&&-\sum_{n=1}^\infty 
\frac{\sum_{k=0}^{b-1}\cos\left(2\pi\left(\frac{k}{b}\right)n\right)\left\{\frac
{ak}{b}\right\}_2\sum_{l=0}^{c-1}\sin\left(2\pi\left(\frac{l}{c}\right)n\right)
-\sum_{k=0}^{b-1}\cos\left(2\pi\left(\frac{k}{b}\right)n\right)\sum_{l=0}^{c-1}
\sin\left(2\pi\left(\frac{l}{c}\right)n\right)\left\{\frac{al}{c}\right\}_2}{n^2
} \\
\hspace{-.5cm}&=& \hspace{-.2cm} \sum_{n=1}^\infty 
\frac{\sum_{k=0}^{b-1}\sin\left(2\pi\left(\frac{k}{b}\right)n\right)\left\{\frac
{ak}{b}\right\}_2\sum_{l=0}^{c-1}\cos\left(2\pi\left(\frac{l}{c}\right)n\right)
+
\sum_{k=0}^{b-1}\cos\left(2\pi\left(\frac{k}{b}\right)n\right)\sum_{l=0}^{c-1}
\sin\left(2\pi\left(\frac{l}{c}\right)n\right)\left\{\frac{al}{c}\right\}_2}{n^2
}. \\
\end{eqnarray*}
We evaluate
$\sum_{k=0}^{b-1}\sin\left(2\pi\left(\frac{k}{b}\right)n\right)\left\{\frac{ak}{
b}\right\}_2$. If $b|n$ we get zero. If not, we apply the Fourier series for
$\{\cdot\}_2$ and obtain
\begin{eqnarray*}
\sum_{k=0}^{b-1}\sin\left(2\pi\left(\frac{k}{b}\right)n\right)\left\{\frac{ak}{b
}\right\}_2&=&-\frac{1}{\pi}\sum_{k=0}^{b-1}\sin\left(2\pi\left(\frac{k}{b}
\right)n\right)\sum_{h=1}^\infty \frac{\sin(2\pi\frac{ak}{b}  h)}{h}\\
&=&  -\frac{1}{\pi}\sum_{h=1}^\infty
\frac{\sum_{k=0}^{b-1}\sin\left(2\pi\left(\frac{k}{b}
\right)n\right)\sin(2\pi\frac{ak}{b}  h)}{h}\\
&=&  -\frac{1}{\pi}\sum_{h=1}^\infty \frac{\sum_{k=0}^{b-1}(\cos\left(\frac{2\pi
k}{b}(n-ah)\right)-\cos\left(\frac{2\pi k}{b}(n+ah)\right))}{2h}.\\
\end{eqnarray*}
The inner finite sum is different from zero only if $b\mid (n-ah)$ or $b\mid
(n+ah)$, in other words, $ah\equiv \pm n \,(\mod\, b)$. Notice that this is only
possible if $(a,b)\mid n$.  Thus, we assume that $n=(a,b)m$. We write this as
$h=\pm [a_b^{-1}m]_{b_a} +rb_a$ where $r$ is an integer that is either
nonnegative or positive depending on the sign for the first term.  Thus we get 
\begin{eqnarray*}
\sum_{{k=0}\atop{b\nmid n, (a,b)\mid
n}}^{b-1}\sin\left(2\pi\left(\frac{k}{b}\right)n\right)\left\{\frac{ak}{b}
\right\}_2&=&-\frac{b}{2\pi}\left(\frac{1}{[a_b^{-1}m]_{b_a}} +\sum_{r=1}^\infty
\frac{1}{rb_a+[a_b^{-1}m]_{b_a}}-\frac{1}{rb-[a_b^{-1}m]_{b_a}} \right) \\
&=&-\frac{(a,b)}{2\pi}\left(\frac{b_a}{[a_b^{-1}m]_{b_a}}
+2\frac{[a_b^{-1}m]_{b_a}}{b_a}\sum_{r=1}^\infty
\frac{1}{\frac{[a_b^{-1}m]_{b_a}^2}{b_a^2}-r^2}\right) \\
&=&-\frac{(a,b)}{2} \cot \left(\pi \frac{[a_b^{-1}m]_{b_a}}{b_a} \right).\\
\end{eqnarray*}

Putting all of the above together for the terms with $S_2$, we obtain
\[\sum_{k=0}^{b-1}
\sum_{l=0}^{c-1}S_2\left(2\pi\left(\frac{k}{b}-\frac{l}{c}\right)\right) 
\left(\left\{\frac{ak}{b}\right\}_2-\left\{\frac{al}{c}\right\}_2\right)\]
\begin{eqnarray} \label{eq:s2}
\hspace{-.3cm} &\hspace{-.3cm} =&  -\frac{c}{2(a,b)}\sum_{{m=1}\atop{c\mid m,
b_a\nmid m, }}^\infty  \frac{\cot \left(\pi \frac{[a_b^{-1}m]_{b_a}}{b_a}
\right)}{m^2}-\frac{b}{2(a,c)}\sum_{{m=1}\atop{b\mid m,  c_a\nmid m, }}^\infty
\frac{\cot \left(\pi \frac{[a_c^{-1}m]_{c_a}}{c_a} \right) }{m^2}. \nonumber\\
\end{eqnarray}
We now write $m=ch$ in the first term and $m=bh$ in the second term. This can be
done since $(a,b,c)=1$. Then
\[\frac{[a_b^{-1}m]_{b_a}}{b_a}= \frac{[a_b^{-1}]_{b_a} ch}{b_a}\]
and analogously in the second term. Thus equation \eqref{eq:s2} equals
\begin{eqnarray*}
&& \hspace{-.3cm} -\frac{1}{2c(a,b)}\sum_{{h=1}\atop{b_a\nmid ch}}^\infty 
\frac{\cot \left(\pi \frac{[a_b^{-1}]_{b_a}ch}{b_a}
\right)}{h^2}-\frac{1}{2b(a,c)}\sum_{{h=1}\atop{ c_a\nmid bh}}^\infty \frac{\cot
\left(\pi \frac{[a_c^{-1}]_{c_a}bh}{c_a} \right) }{h^2}. \\
\end{eqnarray*}

Finally, we get
\[-2m(x^a-1,x^b-1,x^c-1)\]
\begin{eqnarray*}
&=& abc\left(\frac{1}{[a,b]^3}+\frac{1}{[b,c]^3}+\frac{1}{[c,a]^3}\right)
\zeta(3)\\
&&-\frac{\pi}{2c(a,b)}\sum_{{h=1}\atop{b_a\nmid ch}}^\infty  \frac{\cot
\left(\pi \frac{[a_b^{-1}]_{b_a}ch}{b_a}
\right)}{h^2}-\frac{\pi}{2b(a,c)}\sum_{{h=1}\atop{ c_a\nmid bh}}^\infty
\frac{\cot \left(\pi \frac{[a_c^{-1}]_{c_a}bh}{c_a} \right) }{h^2} \\
&&-\frac{\pi}{2a(b,c)}\sum_{{h=1}\atop{c_b\nmid ah}}^\infty  \frac{\cot
\left(\pi \frac{[b_c^{-1}]_{c_b}ah}{c_b}
\right)}{h^2}-\frac{\pi}{2c(b,a)}\sum_{{h=1}\atop{ a_b\nmid ch}}^\infty
\frac{\cot \left(\pi \frac{[b_a^{-1}]_{a_b}ch}{a_b} \right)}{h^2}\\ 
&&-\frac{\pi}{2b(c,a)}\sum_{{h=1}\atop{a_c\nmid bh}}^\infty  \frac{\cot
\left(\pi \frac{[c_a^{-1}]_{a_c}bh}{a_c}
\right)}{h^2}-\frac{\pi}{2a(c,b)}\sum_{{h=1}\atop{b_c\nmid ah}}^\infty 
\frac{\cot \left(\pi \frac{[c_b^{-1}]_{b_c}ah}{b_c} \right)}{h^2}.\\
\end{eqnarray*}
This concludes the proof of Theorem \ref{thm:m3}.
\end{proof}

We can immediately deduce some particular formulae.
\begin{coro}\label{particular-cases}
\begin{enumerate}
\item For positive integers $a$ and $b$ with $(a,b)=1$,
\begin{eqnarray*}
m(x^a-1,x^b-1,x^b-1)&=&
-\frac{2+a^3}{2a^2b}\zeta(3)+\frac{\pi}{2b}\sum_{{h=1}\atop{ a\nmid h}}^\infty
\frac{\cot \left(\pi \frac{h}{a} \right)}{h^2}.
\end{eqnarray*}
\item For an odd integer $d,$ we have
\[m(x-1,x^4-1,x^{2d}-1)\]
\begin{eqnarray*}
&=& -\frac{9+d^3}{16d^2} \zeta(3)+\frac{\pi}{16}\sum_{{h=1}\atop{ d\nmid
2h}}^\infty \frac{\cot \left(\pi \frac{2h}{d} \right) }{h^2}
+\frac{\pi}{8}\sum_{{h=1}\atop{d\nmid h}}^\infty  \frac{\cot \left(\pi
\frac{(d+1)h}{2d} \right)}{h^2}.\\
\end{eqnarray*}
\end{enumerate}
\end{coro}
Here are some particular cases
\begin{eqnarray*}
m(x-1,x^b-1,x^b-1) &=& -\frac{3}{2b}\zeta(3),\\
m(x^2-1,x^b-1,x^b-1) &=& -\frac{5}{4b}\zeta(3),\\
m(x^3-1,x^b-1,x^b-1) &=&
-\frac{29}{18b}\zeta(3)+\frac{\pi}{2\sqrt{3}b}L(2,\chi_{-3}),\\
m(x^4-1,x^b-1,x^b-1) &=& -\frac{33}{16b}\zeta(3)+\frac{\pi}{2b}L(2,\chi_{-4}).\\
\end{eqnarray*}
Here $L(s,\chi)$ denotes the Dirichlet $L$-series in the corresponding character
$\chi$, i.e., $L(s,\chi)=\sum_{n=1}^\infty \frac{\chi(n)}{n^s}$.

\section{Limiting values for $m_k$} \label{sec:limit}

In \cite{B}, Boyd suggests a different point of view for the study of Lehmer's
question. He proposes the study of the set 
\[L=\{m(P)\,:\, P \mbox{ univariate with integer coefficients}\} \subset [0,\infty).\]
(Boyd writes this in terms of the Mahler measure $M(P)=e^{m(P)}$ but we will
keep everything in terms of the logarithmic Mahler measure for consistency.) The
idea is that Lehmer's question can be translated as whether 0 is a limit point
of $L$. In fact, as Boyd points out, if 0 is a limit point of $L$, then $L$ is
dense in $[0,\infty)$. A negative answer to Lehmer's question yields a much more
interesting $L$. Presumably, $L$ is not closed, since $L$ consists of logarithms
of algebraic numbers, but $z_0=\frac{7}{2\pi^2}\zeta(3)$ is a limit point of $L$
and we do not expect  $z_0$ to be the logarithm of an algebraic number. If the
above is true and if Lehmer's question has a negative answer, then one could ask
about other limit points for $L$.

In this section, we proceed to study limits of some sequences in 
\[L_{2h+1}=\{m_{2h+1}(P)\,:\, P \mbox{ univariate with integer coefficients}\},\]
with special focus on 0 as a limit point. Namely, we will show that we can
obtain certain values (including 0) as limit of sequences $\{m_{2h+1}(P_n)\}_n$
where $P_n \in \Z[x]$. 
 
By a generalization of a result of Boyd and Lawton (Theorem \ref{superBL}), $m_k$ of any multivariate polynomial is a limit of a sequence of $m_k$ of univariate polynomials. Therefore, the set
\[L_{2h+1}^\#=\{m_{2h+1}(P)\,:\, P \mbox{ multivariate with integer coefficients}\},\]
is included in the closure of $L_{2h+1}$. We will see that Lehmer's question has a positive answer for $m_{2h+1}$ for $h \geq 1$. Thus, following Boyd, $L_{2h+1}^\#$ can not be a closed set. 

\subsection{Limiting values for $m_3$}\label{BLm_3}

In order to find limit points of $m_3$ of certain sequences of polynomials, we
will need the following result.

\begin{lem} \begin{enumerate}
\item Let $r \in \Z$, $r\not = 0$ and $p \in \Z$. Then
 \[\lim_{p\rightarrow \infty} \frac{r\pi}{p}\sum_{{h=1}\atop{ p\nmid rh}}^\infty
\frac{\cot \left(\pi \frac{rh}{p} \right)}{h^2} = \zeta(3).\]
\item Let $p \in \Z$ be odd. Then
\[\lim_{p \rightarrow \infty}\frac{4\pi}{p}\sum_{{h=1}\atop{p\nmid h}}^\infty 
\frac{\cot \left(\pi \frac{(p+1)h}{2p} \right)}{h^2}=\zeta(3).\]
\end{enumerate}
\end{lem}

\begin{proof} Part {\em 1}.
Observe that $\cot(x) < \frac{1}{x}$ for $0<x< \pi$. Thus, for $0< h <
\frac{p}{r}$, we can write
 \[\cot \left(\pi \frac{rh}{p} \right) < \frac{p}{rh\pi}.\]
Moreover, for $p\nmid h$, we have  that 
\[\left|\cot \left(\pi \frac{rh}{p} \right)\right|<\frac{p}{\pi}.\]
Thus,
\begin{eqnarray*}
\frac{r\pi}{p}\sum_{{h=1}\atop{ p\nmid rh}}^\infty \frac{\cot \left(\pi
\frac{rh}{p} \right)}{h^2} &=& \frac{r\pi}{p}\sum_{1\leq h < \frac{p}{r}}
\frac{\cot \left(\pi \frac{rh}{p}
\right)}{h^2}+\frac{r\pi}{p}\sum_{{\frac{p}{r}<h}\atop{ p\nmid rh}} \frac{\cot
\left(\pi \frac{rh}{p} \right)}{h^2}\\
&< & \sum_{1\leq h < \frac{p}{r}} \frac{1}{h^3}+r \sum_{{\frac{p}{r}<h}}
\frac{1}{h^2}
\end{eqnarray*}

On the other hand, $\lim_{x\rightarrow 0} x\cot(x)=1$. Given $\epsilon>0$, take
$p$ large enough such that 
\[\cot \left(\pi \frac{rh}{p} \right) \geq \frac{p}{rh \pi }(1-\epsilon)\]
for any $0<h<\sqrt{p}$. Let $H=\lfloor \sqrt{p}\rfloor$. Then 
\begin{eqnarray*}
 \frac{r\pi}{p}\sum_{{h=1}\atop{ p\nmid rh}}^\infty \frac{\cot \left(\pi
\frac{rh}{p} \right)}{h^2} &=& \frac{r\pi}{p}\sum_{{h=1}\atop{ p\nmid rh}}^{H}
\frac{\cot \left(\pi \frac{rh}{p} \right)}{h^2}+\frac{r\pi}{p}\sum_{{h=H}\atop{
p\nmid rh}}^\infty \frac{\cot \left(\pi \frac{rh}{p} \right)}{h^2}\\
&\geq & (1-\epsilon)\sum_{{h=1}}^{H}
\frac{1}{h^3}-\frac{r\pi}{p}\sum_{{h=H}\atop{ p\nmid rh}}^\infty
\frac{\left|\cot \left(\pi \frac{rh}{p} \right)\right|}{h^2}\\
&\geq & (1-\epsilon)\sum_{{h=1}}^{H} \frac{1}{h^3}-r\sum_{{h=H}}^\infty
\frac{1}{h^2}\\
\end{eqnarray*}
Taking the limit when $p \rightarrow \infty$, we conclude that
\[\lim_{p\rightarrow \infty} \frac{r\pi}{p}\sum_{{h=1}\atop{ p\nmid rh}}^\infty
\frac{\cot \left(\pi \frac{rh}{p} \right)}{h^2} = \zeta(3).\]

Part {\em 2}.
\begin{equation}\label{eq:cot4}
\frac{4\pi}{p}\sum_{{h=1}\atop{p\nmid h}}^\infty  \frac{\cot \left(\pi
\frac{(p+1)h}{2p} \right)}{h^2} = \frac{4\pi}{p}\sum_{{h=1}\atop{2|h, p\nmid
h}}^\infty  \frac{\cot \left(\pi \frac{(p+1)h}{2p}
\right)}{h^2}+\frac{4\pi}{p}\sum_{{h=1}\atop{2 \nmid h, p\nmid h}}^\infty 
\frac{\cot \left(\pi \frac{(p+1)h}{2p} \right)}{h^2}.
\end{equation}
For the first term, we let $h=2j$. For the second term, we observe that, for
$0<h<p$, \[\cot\left(\pi\frac{(p+1)h}{2p}\right)= \cot \left(
\frac{\pi}{2}+\frac{h\pi}{2p}\right)=\cot \left( \pi\frac{p+h}{2p}\right)<
\frac{2p}{\pi(p+h)},\]
and for $p\nmid h$,
 \[\left|\cot\left(\pi\frac{(p+1)h}{2p}\right) \right| < \frac{2p}{\pi}.\]

Thus, equation \eqref{eq:cot4} equals
\begin{eqnarray*}
&& \frac{\pi}{p}\sum_{{j=1}\atop{p\nmid j}}^\infty  \frac{\cot \left(\pi
\frac{(p+1)j}{p} \right)}{j^2}+\frac{4\pi}{p}\sum_{1\leq h<p}  \frac{\cot
\left(\pi \frac{(p+1)h}{2p} \right)}{h^2}+ \frac{4\pi}{p}\sum_{{p<h}\atop{p\nmid
h}}^\infty  \frac{\cot \left(\pi \frac{(p+1)h}{2p} \right)}{h^2}\\
&\leq& \frac{\pi}{p}\sum_{{j=1}\atop{p\nmid j}}^\infty  \frac{\cot \left(\pi
\frac{j}{p} \right)}{j^2}+8\sum_{1\leq h<p} \frac{1}{(p+h)h^2} +
8\sum_{{p<h}\atop{p\nmid h}}^\infty \frac{1}{h^2}\\
&\leq& \frac{\pi}{p}\sum_{{j=1}\atop{p\nmid j}}^\infty  \frac{\cot \left(\pi
\frac{j}{p} \right)}{j^2}+\frac{8}{p}\zeta(2)+ 8\sum_{{p<h}\atop{p\nmid
h}}^\infty \frac{1}{h^2}.\\
\end{eqnarray*}
Similarly, we can write 
\begin{equation*}
\frac{4\pi}{p}\sum_{{h=1}\atop{p\nmid h}}^\infty  \frac{\cot \left(\pi
\frac{(p+1)h}{2p} \right)}{h^2} \geq \frac{\pi}{p}\sum_{{j=1}\atop{p\nmid
j}}^\infty  \frac{\cot \left(\pi \frac{j}{p} \right)}{j^2}-\frac{8}{p}\zeta(2)-
8\sum_{{p<h}\atop{p\nmid h}}^\infty \frac{1}{h^2}.
\end{equation*}
By taking the limit when $p \rightarrow \infty$ and using Part {\em 1}, we
conclude the proof. 
\end{proof}

We will now compute $m_3$ for some sequences of polynomials and take their
limits. This process will provide us with limit points for the values of $m_3$
as well as infinitely many polynomials $P$ with positive and negative values of
$m_3(P)$.

\begin{enumerate}

\item Consider the family of polynomials $\frac{x^p-1}{x-1}$. From part 1 of
Corollary \ref{particular-cases}, we have that
\begin{eqnarray*}
m_3\left(\frac{x^p-1}{x-1} \right)&=&m_3(x^p-1)-m_3(x-1)+3m(x^p-1,x-1,x-1)\\
&-&3m(x^p-1,x^p-1,x-1)\\
&=&3\left(-\frac{2+p^3}{2p^2}\zeta(3)+\frac{\pi}{2}\sum_{{h=1}\atop{ p\nmid
h}}^\infty \frac{\cot \left(\pi \frac{h}{p} \right)}{h^2}  +
\frac{3}{2p}\zeta(3)\right)\\
&=&\frac{9p-6-3p^3}{2p^2}\zeta(3)+\frac{3\pi}{2}\sum_{{h=1}\atop{ p\nmid
h}}^\infty \frac{\cot \left(\pi \frac{h}{p} \right)}{h^2}. \\
\end{eqnarray*}

Thus,
\begin{eqnarray*}
\lim_{p\rightarrow \infty} m_3\left(\frac{x^p-1}{x-1} \right)
&=&\lim_{p\rightarrow \infty}
\left(\frac{9p-6-3p^3}{2p^2}\zeta(3)+\frac{3\pi}{2}\sum_{{h=1}\atop{ p\nmid
h}}^\infty \frac{\cot \left(\pi \frac{h}{p} \right)}{h^2}\right) \\
&=&\lim_{p\rightarrow \infty} \left(\frac{9p-6-3p^3}{2p^2}\zeta(3)+\frac{3
p}{2}\zeta(3)\right) \\
&=&0.
\end{eqnarray*}
Thus, $0$ seems to be a limit point for $L_3$. 

\item Now, let us focus on the case of $(x^p-1)(x-1)$. Again, we apply part 1 of
Corollary \ref{particular-cases}, in order to obtain
\begin{eqnarray*}
m_3\left((x^p-1)(x-1) \right)&=&m_3(x^p-1)+m_3(x-1)+3m(x^p-1,x-1,x-1)\\
&+&3m(x^p-1,x^p-1,x-1) \\
&=&3\left(-\zeta(3)-\frac{2+p^3}{2p^2}\zeta(3)+\frac{\pi}{2}\sum_{{h=1}\atop{
p\nmid h}}^\infty \frac{\cot \left(\pi \frac{h}{p} \right)}{h^2}  -
\frac{3}{2p}\zeta(3)\right)\\
&=&\frac{-6p^2-9p-6-3p^3}{2p^2}\zeta(3)+\frac{3\pi}{2}\sum_{{h=1}\atop{ p\nmid
h}}^\infty \frac{\cot \left(\pi \frac{h}{p} \right)}{h^2}. \\
\end{eqnarray*}

Thus,
\begin{eqnarray*}
\lim_{p\rightarrow \infty} m_3\left((x^p-1)(x-1)\right)
&=&\lim_{p\rightarrow \infty}
\left(\frac{-6p^2-9p-6-3p^3}{2p^2}\zeta(3)+\frac{3\pi}{2}\sum_{{h=1}\atop{
p\nmid h}}^\infty \frac{\cot \left(\pi \frac{h}{p} \right)}{h^2}\right) \\
&=&\lim_{p\rightarrow \infty} \left(\frac{-6p^2-9p-6-3p^3}{2p^2}\zeta(3)+\frac{3
p}{2}\zeta(3)\right) \\
&=&-3\zeta(3).
\end{eqnarray*}
Thus, $-3\zeta(3)$ seems to be a limit point for $L_3$. In addition, we obtain
infinitely many polynomials $P$ such that $m_3(P)<0$.

\item We now look at the case $a=1$, $b=4$ and $c=2d$ with $d$ odd.  Applying
part 2 of Corollary \ref{particular-cases} and observing that for an odd integer
$d,$ $[2^{-1}]_{d}=\frac{d+1}{2},$ we get
\begin{eqnarray*}
m_3\left(\frac{(x^4-1)(x^{2d}-1)}{(x-1)^2}\right)
&=&m_3(x^4-1)+m_3(x^{2d}-1)-8m_3(x-1)\\
&&+3m(x^4-1,x^{2d}-1,x^{2d}-1)\\
&&+3 m(x^4-1,x^4-1,x^{2d}-1)\\
&&+12m(x^4-1,x-1,x-1)\\
&&-6m(x^4-1,x^4-1,x-1)\\
&&+12m(x^{2d}-1,x-1,x-1)\\
&&-6m(x^{2d}-1,x^{2d}-1,x-1)\\
&&-12m(x^4-1,x^{2d}-1,x-1)\\
&=&-\frac{3}{2} \zeta(3)-\frac{3}{2} \zeta(3)+12 \zeta(3)\\
&& -\frac{15}{4d}\zeta(3)
-\frac{6+3d^3}{4d^2}\zeta(3)+\frac{3\pi}{4}\sum_{{h=1}\atop{ d\nmid h}}^\infty
\frac{\cot \left(\pi \frac{h}{d} \right)}{h^2}\\
&&-\frac{99}{4}\zeta(3)+6\pi L(2,\chi_{-4})+\frac{9}{4}\zeta(3)\\
&&-\frac{3+12d^3}{d^2}\zeta(3)+6\pi \sum_{{h=1}\atop{ 2d\nmid h}}^\infty
\frac{\cot \left(\pi \frac{h}{2d} \right)}{h^2}+\frac{9}{2d}\zeta(3)\\
&&+\frac{27+3d^3}{4d^2} \zeta(3)-\frac{3\pi}{4}\sum_{{h=1}\atop{ d\nmid
2h}}^\infty \frac{\cot \left(\pi \frac{2h}{d} \right) }{h^2}\\
&& -\frac{3\pi}{2}\sum_{{h=1}\atop{d\nmid h}}^\infty  \frac{\cot \left(\pi
\frac{(d+1)h}{2d} \right)}{h^2}\\ 
&=&\frac{9+3d-54d^2-48d^3}{4d^2}\zeta(3)+6\pi L(2,\chi_{-4})\\
&& +\frac{3\pi}{4}\sum_{{h=1}\atop{ d\nmid h}}^\infty \frac{\cot \left(\pi
\frac{h}{d} \right)}{h^2}+6\pi \sum_{{h=1}\atop{ 2d\nmid h}}^\infty \frac{\cot
\left(\pi \frac{h}{2d} \right)}{h^2}\\
&&-\frac{3\pi}{4}\sum_{{h=1}\atop{ d\nmid 2h}}^\infty \frac{\cot \left(\pi
\frac{2h}{d} \right) }{h^2} 
-\frac{3\pi}{2}\sum_{{h=1}\atop{d\nmid h}}^\infty  \frac{\cot \left(\pi
\frac{(d+1)h}{2d} \right)}{h^2}.\\ 
\end{eqnarray*}

Thus
\begin{eqnarray*}
 \lim_{d \rightarrow \infty}
m_3\left(\frac{(x^4-1)(x^{2d}-1)}{(x-1)^2}\right)&=& \lim_{d \rightarrow \infty}
\left( \frac{9+3d-54d^2-48d^3}{4d^2}\zeta(3)\right.\\
 &&\left.+6\pi L(2,\chi_{-4})
  +\frac{3\pi}{4}\sum_{{h=1}\atop{ d\nmid h}}^\infty \frac{\cot \left(\pi
\frac{h}{d} \right)}{h^2}\right.\\
&&\left. +6\pi \sum_{{h=1}\atop{ 2d\nmid h}}^\infty \frac{\cot \left(\pi
\frac{h}{2d} \right)}{h^2}-\frac{3\pi}{4}\sum_{{h=1}\atop{ d\nmid h}}^\infty
\frac{\cot \left(\pi \frac{2h}{d} \right) }{h^2}\right.\\
&&\left. -\frac{3\pi}{2}\sum_{{h=1}\atop{d\nmid h}}^\infty  \frac{\cot \left(\pi
\frac{(d+1)h}{2d} \right)}{h^2}\right)\\
&=& \lim_{d \rightarrow \infty} \left(
\frac{9+3d-54d^2-48d^3}{4d^2}\zeta(3)\right.\\
&&\left.+6\pi L(2,\chi_{-4}) +\frac{3 d}{4}\zeta(3)\right.\\
&&\left. +12 d \zeta(3)-\frac{3d}{8}\zeta(3) -\frac{3d}{8}\zeta(3)\right)\\
&=&6\pi L(2,\chi_{-4})-\frac{27}{2}\zeta(3)\cong 1.0377764969\dots.
\end{eqnarray*}
Thus, $6\pi L(2,\chi_{-4})-\frac{27}{2}\zeta(3)$ seems to be a limit point for
$L_3$. In addition, we obtain infinitely many polynomials $P$ such that
$m_3(P)>0$.

\item It is not generally hard to find positive limit points for $m_3(P)$, for
example, one can take the sequence $(x^n+3)(x+3)$. It is clear that
$m_3((x^n+3)(x+3))\geq \log^3 4>0$.

\end{enumerate}

\subsection{Limit values for higher Mahler measures} \label{sec:Boyd-Lawton}

Analogously to the Mahler measure for one variable, the Mahler measure of a
non-zero multi-variable polynomial $P(x_1,\dots,x_n) \in \C[x_1,\dots,x_n]$ can
be defined as
\[m(P) := \frac{1}{(2\pi i)^n}\int_{|x_1|=1}\dots \int_{|x_n|=1} \log
|P(x_1,\dots,x_n)|\frac{dx_1}{x_1}\dots\frac{dx_n}{x_n}.\]

This generalization can be extended to the multiple (and higher) Mahler measure.
Let $P_1,\dots,P_l \in \C[x_1,\dots,x_n]$ be nonzero polynomials. Then, we
define $m(P_1,\dots,P_l)$ as
$$  \frac{1}{(2\pi i)^n}\int_{|x_1|=1}\dots \int_{|x_n|=1} \log
|P_1(x_1,\dots,x_n)|\dots \log
|P_l(x_1,\dots,x_n)|\frac{dx_1}{x_1}\dots\frac{dx_n}{x_n}.$$
Boyd \cite{B1} conjectured the following important statement, which was
completely proved by Lawton \cite{Lawton}.  
\begin{thm}
Let $P(x_1,\dots,x_n) \in \C[x_1,\dots,x_n]$ and ${\bf r} =
(r_1,\dots,r_n),\,r_i \in \Z_{>0}.$  Define $P_{{\bf r}}(x)$ as
\[P_{{\bf r}}(x) = P(x^{r_1},\dots,x^{r_n}),\] and let
\[q({\bf r}) = \min \left \{H({\bf s}): \,{\bf s} = (s_1,\dots,s_n) \in \Z^n,\,
{\bf s} \neq (0,\dots,0), \sum_{j=1}^ns_jr_j = 0 \right\},\]
where $H({\bf s}) = \max\{|s_j|:\,1\leq j \leq n\}$.  Then
\[\lim_{q({\bf r})\to \infty}m(P_{{\bf r}}) = m(P).\]
\end{thm} 

It is a simple exercise to generalize the techniques of Lawton to prove an
analogous result for multiple Mahler measures.  That is, under the same
conditions as above, one can show 
\begin{thm}\label{superBL} Let $P_1,\dots,P_l \in \C[x_1,\dots,x_n],$  and ${\bf
r}$ as before. Then
\[\lim_{q({\bf r})\to \infty}m({P_1}_{{\bf r}},\dots, {P_l}_{{\bf r}}) =
m(P_1,\dots,P_l).\]
\end{thm}

As an immediate application of Theorem \ref{superBL}, we get that for any $a\geq
1,$ 
$$\lim_{p\to \infty} m(x^a-1,x^p-1,x^p-1) = m(x-1)m_2(y-1) = 0$$ 
and
$$\lim_{p\to \infty} m(x^a-1,x^a-1,x^p-1) = m_2(x - 1)m(y-1) = 0.$$

Thus, the limits from Section \ref{BLm_3} follow from this.  An advantage of
Theorem \ref{superBL} over the techniques in Section \ref{BLm_3} is that it
gives us the limits of $m_k$ of these sequences for all values of $k.$  For
example, we immediately obtain that 
\[\lim_{n\rightarrow \infty} m_{2h+1}\left(\frac{x^n-1}{x-1}\right)=0.\]

We will prove in the next subsection that the above sequence (for $h>0$ fixed) is nonconstant. While 0 is a limit point of $m_{2h+1}$, what can be said about
positive and negative values? As in the case of $m_3$, it is not hard to see
that $m_k(x+3)\geq \log^k 2>0$. Using Theorem  \ref{superBL} we can see that the
sequence $m_{2h+1}((x^n+3)(x+3))$ has a positive limit. As for negative limits,
the sequence $m_{2h+1}((x^n-1)(x-1))$ provides a good example. To see this, we
apply the following result from \cite{KLO} (Theorem 3):

\begin{thm} \label{KLO-beauty}
For $l \in \Z_{\geq 1}$,
\[m_l(x-1) = \sum_{b_1+\dots +b_j = l,\,b_i\geq
2}\frac{(-1)^ll!}{2^{2j}}\zeta(b_1,\dots,b_j),\]
where 
\[\zeta(b_1,\dots,b_j) = \sum_{1\leq p_1<\dots<p_j}\frac{1}{p_1^{b_1}\dots
p_j^{b_j}}.\]
\end{thm}

From Theorem \ref{superBL}, we get that
\[\lim_{n\to\infty} m_{2h+1}((x^n-1)(x-1))= m_{2h+1}((y-1)(x-1))\]
\[=\sum_{i=0}^{2h+1} \binom{2h+1}{i} m_i(y-1)m_{2h+1-i}(x-1).\]
Moreover, Theorem \ref{KLO-beauty} tells us that $m_l(x-1) <0$ for odd $l$ and
$m_l(x-1)>0$ for even $l,$ that is, each term on the right hand side of the
above equation is negative.  Thus, $m_{2h+1}((x^n-1)(x-1))$ has a negative
limit.  

On a different note, observe that $\frac{\pi^2}{12}$ is a limiting value for
$m_2$, since, by  Theorem 19 (iv) in \cite{KLO}, we have that
$m_2(x+y+2)=\frac{\pi^2}{12}$. Thus
\[ \lim_{n\rightarrow \infty} m_2(x^n+x+2) =\frac{\pi^2}{12}.\]

\subsection{A proof that certain sequences are nonconstant}
As usual, Theorem \ref{superBL} does not say anything about the sequence of
values $m_{2h+1}\left(\frac{x^n-1}{x-1}\right)$, which in principle could be
constant (and therefore, identically zero). This is precisely the case with $h=0$. 

Fortunately, we have the following result.
 \begin{thm}\label{KS} Let $h\geq 1$ fixed and $P_n(x)=\frac{x^n-1}{x-1}$. Then the sequence $m_{2h+1}(P_n)$ is nonconstant.
 \end{thm}
The idea of this proof was provided to us by Kannan Soundararajan. We will need some auxiliary results first.

\begin{lem} Let $\alpha, m \in \Z$ with $m$ positive.  
 Let \[T_m(\alpha):=\sum_{{\ell_1,\dots,\ell_m \in \Z_{\not =
0}}\atop{\ell_1+\dots+\ell_m=\alpha}}\frac{1}{|\ell_1|\dots|\ell_m| }\] Then, for $\alpha \not= 0$,
\[T_m(\alpha)= \frac{2^{m-1}m\log^{m-1}|\alpha|}{|\alpha|} \left( 1 + O\left(\log^{-2}|\alpha| \right) \right). \]
\end{lem}
\begin{proof}
First notice that $T_m(\alpha)=T_m(-\alpha)$, so we can assume that $\alpha$ is positive.  By multiplying and dividing by $\ell_1+\dots+\ell_m=\alpha$, we obtain that
\begin{eqnarray*}T_m(\alpha)&=&\frac{1}{\alpha}\sum_{{\ell_1,\dots,\ell_m \in \Z_{\not =
0}}\atop{\ell_1+\dots+\ell_m=\alpha}}\frac{\ell_1+\dots+\ell_m}{|\ell_1|\dots|\ell_m| }=\frac{1}{\alpha} \sum_{j=1}^m \sum_{\ell_j \in \Z_{\not = 0}} \mathrm{sign}(\ell_j)T_{m-1}(\alpha-\ell_j)\\
&=&\frac{m}{\alpha}\sum_{\ell\in\Z_{\not = 0}}\mathrm{sign}(\ell)T_{m-1}(\alpha-\ell)\\
&=&\frac{m}{\alpha}\left(-\sum_{\ell = -\infty}^{-1}T_{m-1}(\alpha-\ell) +\sum_{\ell =1}^{\alpha}T_{m-1}(\alpha-\ell) + \sum_{\ell = \alpha+1}^{2 \alpha}T_{m-1}(\alpha-\ell)+ \sum_{\ell= 2 \alpha+1}^\infty T_{m-1}(\alpha-\ell)\right)\\
&=& \frac{2m}{\alpha} \sum_{j=0}^{\alpha} T_{m-1}(j)-\frac{m}{\alpha}(T_{m-1}(0)+T_{m-1}(\alpha)).
\end{eqnarray*}

Now observe that $T_1(\alpha)=\frac{1}{|\alpha|}$ for $\alpha \not = 0$. We proceed by induction. Assume that the statement is true for $m$. Then 
\begin{eqnarray*}
T_{m+1}(\alpha)&=&\frac{2(m+1)}{\alpha} \sum_{j=0}^{\alpha} T_{m}(j)-\frac{m+1}{\alpha}(T_{m}(0)+T_{m}(\alpha))\\
&=& \frac{2(m+1)}{\alpha} \sum_{j=1}^{\alpha} \frac{2^{m-1}m\log^{m-1}j}{j} \left( 1 + O\left(\log^{-2} j \right) \right)-\frac{m+1}{\alpha}\frac{2^{m-1}m\log^{m-1}\alpha}{\alpha}\left( 1 + O\left(\log^{-2}\alpha \right) \right)\\
\end{eqnarray*}
We now replace the above sum with the integral of $\frac{\log^{m-1}x}{x}$ (with exponent $m-3$ for the error term). This replacement introduces another error term of $O \left(\frac{\log^{m-1} \alpha}{\alpha} \right)$.  We deduce that
\begin{eqnarray*}
T_{m+1}(\alpha)
&=& \frac{2^m(m+1)}{\alpha} \left(\log^m\alpha + O \left(\frac{\log^{m-1} \alpha}{\alpha} \right) \right) + O\left(\frac{\log^{m-2}\alpha}{\alpha} \right)-\frac{2^{m-1}m(m+1)\log^{m-1}\alpha}{\alpha^2}\left( 1 + O\left(\log^{-2}\alpha \right) \right)\\
&=& \frac{2^m(m+1)}{\alpha} \log^m\alpha \left( 1 + O\left(\log^{-2} \alpha \right) \right).\\
\end{eqnarray*}
\end{proof}
\begin{prop}\label{U} Let $j,k \in \Z_{\geq 1}$. There is a positive constant $C(j,k)$ such that
\[U_{j,k}^{(n)}:=\sum_{{\ell_1,\dots, \ell_{j+k} \in \Z_{\not =
0}}\atop{\ell_1+\dots+\ell_j+n\ell_{j+1}+\dots+n\ell_{j+k}=0}}\frac{1}{|\ell_1|\dots|\ell_{j+k}| }=C(j,k) \frac{\log^{j-1}n}{n}\left(1 + O\left( \log^{-1}n \right) \right).\]
\end{prop}
\begin{proof} We have that
\begin{eqnarray*}
U_{j,k}^{(n)}&=&\sum_{\alpha\in \Z} T_j(n\alpha)T_k(-\alpha)\\
&=& 2 \sum_{\alpha=1}^\infty \frac{2^{j-1}j\log^{j-1}|n\alpha|}{|n\alpha|} \left( 1 + O\left(\log^{-2}|n\alpha| \right) \right)T_k(\alpha).\\
\end{eqnarray*}
We only need to study the behavior when $n$ goes to infinity. Therefore, we do not need to have $\alpha$ in the error term. We write $\log^{j-1}|n \alpha|= \log^{j-1}|n| + O(\log^{j-2}|n|)$ and we obtain
\begin{eqnarray*}
U_{j,k}^{(n)}&=& \frac{2^jj\log^{j-1}n}{n} \sum_{\alpha=1}^\infty \frac{T_k(\alpha)}{\alpha} \left(1 + O\left( \log^{-1}n \right) \right).
\end{eqnarray*}
Notice that $T_k(\alpha)>0$ by construction, and so is $C(j,k)$.
\end{proof}

\begin{proof}[Theorem \ref{KS}]
By writing the integral and using Fourier expansions, we obtain
\begin{eqnarray*}
m_{2h+1}(P_n) &=& \int_0^1 \log^{2h+1}\left|\frac{e^{2\pi i n\theta}-1}{e^{2\pi i \theta}-1} \right| d \theta\\
&=& \sum_{j=0}^{2h+1} \binom{2h+1}{j} (-1)^{j}\int_0^1 \log^j\left|{e^{2\pi i \theta}-1}\right| \log^{2h+1-j}\left|{e^{2\pi i n \theta}-1} \right| d \theta\\
&=& \sum_{j=0}^{2h+1} \binom{2h+1}{j} (-1)^{j}\int_0^1\left(-\frac{1}{2} \sum_{\ell_1\in \Z_{\not = 0}} \frac{e^{2\pi i \ell_1 \theta}} {|\ell_1|} \right)^j \left(-\frac{1}{2} \sum_{\ell_2\in \Z_{\not = 0}} \frac{e^{2\pi i n \ell_2 \theta}} {|\ell_2|} \right)^{2h+1-j} d \theta\\
&=& \sum_{j=0}^{2h+1} \binom{2h+1}{j} \frac{(-1)^{j+1}}{2^{2h+1}}U_{j,2h+1-j}^{(n)}.\\
\end{eqnarray*}
By Proposition \ref{U}, the term with the highest weight in $n$ is for $j=2h$. Notice that the condition $h\geq 1$ is necessary because otherwise we obtain a formula that does not depend on $n$. Thus, we have
\begin{eqnarray*}
m_{2h+1}(P_n)  &=& -\frac{(2h+1)}{2^{2h+1}} C(2h,1) \frac{\log^{2h-1}n}{n}\left(1 + O\left( \log^{-1}n \right) \right).\\
\end{eqnarray*}
Therfore, $m_{2h+1}(P_n)$ behaves like a nonzero constant times $\frac{\log^{2h-1}n}{n}$ when $n$ goes to infinity. This implies that the sequence can not be identically zero.
\end{proof}

The discussion in this section proves Theorem \ref{m_3-limit-0}.
\section{Discussion on the values of $m_k(P)$} \label{sec:table}
We will once again focus our attention on the set
\[L_{k}=\{m_{k}(P)\,:\, P \mbox{ univariate with integer coefficients}\}.\]
For $k=2$, we have
\[L_2=\{m_2(P)\,:\, P \mbox{ univariate with integer coefficients}\} \subset
\left[\frac{\pi^2}{48},\infty\right).\]
In this context, the first noticeable difference between $m(P)$ and $m_2(P)$ is
that the cyclotomic polynomials are interesting in terms of $m_2(P)$. We have
explored this phenomenon in this note. Many questions remain, however, and in
particular, the question of what happens with the reciprocal noncyclotomic
polynomials -the ones that are interesting in the case of the classical Mahler
measure- is presumably as interesting and difficult as in the case of the
classical Mahler measure. In particular, equation \eqref{eq:m_2versusm} and
Proposition \ref{even-higher-Mahler} suggest that a natural object to study is
$m_2(P)-m(P)^2$.

The following table records the noncyclotomic polynomials of degree less or
equal than 14 with $m(P)<0.25$. The data has been obtained from the generator in
Mossinghoff's website \cite{M}.  We observe that the smallest polynomial (in the
table) in terms of $m_2(P)$ is not the degree-10 polynomial of Lehmer, but
$x^{10}+x^9-x^5+x+1$. In fact, {\em all} the polynomials in the table have
$m_2(P)$ smaller than Lehmer's polynomial. This result comes from the fact that
the term $m(P)^2$ in equation \eqref{eq:m_2versusm} seems considerably smaller
than the other terms, and therefore, the contribution of $m(P)$ to the value of
$m_2(P)$ is relatively small for polynomials of small $m(P)$.

\begin{center}
\renewcommand{\arraystretch}{1.3}

\hspace{-2cm}\begin{tabular}{|c|c|c|}

 \hline

$P(x)$ & $m(P)$ & $m_2(P)$\\

\hline

$x^8+x^5-x^4+x^3+1 $& 0.2473585132& 1.0980813745\\

$x^{10}+x^9-x^7-x^6-x^5-x^4-x^3+x+1 $ & 0.1623576120&1.7447964556\\

$x^{10}-x^6+x^5-x^4+1$&0.1958888214&1.2863292447
\\

$x^{10}+x^7+x^5+x^3+1 $&0.2073323581&1.2320444893\\

$x^{10}-x^8+x^5-x^2+1 $&0.2320881973 	 &1.1704950485
\\

$x^{10}+x^8+x^7+x^5+x^3+x^2+1$&0.2368364616&1.1914083866
\\

$x^{10}+x^9-x^5+x+1 $&0.2496548880&1.0309287773\\

$x^{12}+x^{11}+x^{10}-x^8-x^7-x^6-x^5-x^4+x^2+x+1 $&0.2052121880&1.4738375004
\\

$x^{12}+x^{11}+x^{10}+x^9-x^6+x^3+x^2+x+1 $&0.2156970336&1.5143823478\\

$x^{12}+x^{11}-x^7-x^6-x^5+x+1$&0.2239804947&1.2059443050
\\ 

$x^{12}+x^{10}+x^7-x^6+x^5+x^2+1 $&0.2345928411 &1.2434560052
\\

$x^{12}+x^{10}+x^9+x^8+2x^7+x^6+2x^5+x^4+x^3+x^2+1$ &0.2412336268&1.6324129051
\\

$x^{14}+x^{11}-x^{10}-x^7-x^4+x^3+1$&	 0.1823436598&1.3885013172
\\

$x^{14}-x^{12}+x^7-x^2+1 $&0.1844998024&1.3845721865
\\

$x^{14}-x^{12}+x^{11}-x^9+x^7-x^5+x^3-x^2+1 $&0.2272100851& 1.4763006621
\\

$x^{14}+x^{11}+x^{10}+x^9+x^8+x^7+x^6+x^5+x^4+x^3+1 $&	
0.2351686174&1.4352060397
\\

$x^{14}+x^{13}-x^8-x^7-x^6+x+1 $& 0.2368858459 &1.2498299096
\\

$x^{14}+x^{13}+x^{12}-x^9-x^8-x^7-x^6-x^5+x^2+x+1 $&0.2453300143 	
&1.3362661982
 \\

$x^{14}+x^{13}-x^{11}-x^7-x^3+x+1 $&0.2469561884& 1.3898540050
\\

\hline

\end{tabular}

\end{center}
Analogously, we can translate the speculations about $L_2$ to the case of
$L_{2h}$ with $h>1$, a set that  satisfies $L_{2h}\subset
\left[\left(\frac{\pi^2}{48}\right)^h,\infty\right)$.

On the other hand, we have proved that $L_{2h+1}$ (for $h>0$)
has positive and negative values. By taking powers, it is easy to build
sequences of polynomials whose $m_{2h+1}$ tend to either $\infty$ or $-\infty$.
We have also seen that 0 is a limit point. Notice that this
last fact is related to $m_{2h+1}$ being nontrivial on cyclotomic
polynomials, something that is not true in the case of the classical Mahler
measure. 

In conclusion, we see that $m_k(P)$ has very different behavior depending on the
parity of $k$. We expect that $m_k(P)$ for $k>1$ is nontrivial for cyclotomic
polynomials, and that this fact answers Lehmer's question for $k>1$.

\begin{acknowledgements}
We would like to thank David Boyd for his feedback on this work and Kannan
Soundararajan for his interest and his ideas on how to prove that sequences of values of high
Mahler measures are not identically zero. Finally we would like to thank the
referee for many helpful suggestions that have greatly improved the exposition
of this note.
\end{acknowledgements}

\end{document}